\pdfoutput=1

\documentclass{amsart}
\usepackage[a4paper, left=30mm,right=30mm,top=35mm,bottom=35mm]{geometry}
\usepackage{amssymb}
\usepackage{amscd}  
\usepackage{amsxtra}
\usepackage{amsmath}
\usepackage[mathscr]{eucal}
\usepackage{color}
\usepackage{hyperref} %this produces hyperlinks
\hypersetup{
colorlinks=true,
linkcolor=blue,
}

\newtheorem{thm}{Theorem}[section]
\newtheorem{coro}[thm]{Corollary}
\newtheorem{prop}[thm]{Proposition}
\newtheorem{lemma}[thm]{Lemma}
\newtheorem{conjecture}[thm]{Conjecture}
\newtheorem{question}[thm]{Question}
\newtheorem{hypothesis}[thm]{Hypothesis}
\theoremstyle{definition}
\newtheorem{definition}[thm]{Definition}
\newtheorem{remark}[thm]{Remark}
\theoremstyle{remark}

\DeclareMathOperator{\Cl}{Cl}

\DeclareMathOperator{\Gal}{Gal}

\newcommand{\CC}{\mathbb{C}}

\newcommand{\QQ}{\mathbb{Q}}

\newcommand{\RR}{\mathbb{R}}

\newcommand{\ZZ}{\mathbb{Z}}
\newcommand{\Z}{\mathbb{Z}}

\def\bigcapp{\raise1ex\hbox{\rotatebox{180}{$\biguplus$}}}
\def\bigcappd{\raise1ex\hbox{\rotatebox{180}{$\displaystyle\biguplus$}}}

\begin{document}
	\title[Galois-Gauss sums and the square root of the inverse different]{On Galois-Gauss sums and\\ the square root of the inverse different}
	\author{Yu Kuang}
	\date{19 October 2022}
	\email[Y. Kuang]{yu.kuang.3@gmail.com}
	\address{Zhuhai, China}

\begin{abstract}
We discuss a possible generalisation of a conjecture of Bley, Burns and Hahn \cite{BBH} concerning the relation between the second Adams-operator twisted Galois-Gauss sums of weakly ramified Artin characters and the square root of the inverse different of finite, odd degree, Galois extensions of number fields, to the setting of all finite Galois extensions of number fields for which a square root of the inverse different exists. We also extend the key methods and results of \cite{BBH} to this more general setting and, by combining these methods with a recent result of Agboola, Burns, Caputo and the present author concerning Artin root numbers of twisted irreducible symplectic characters, we provide new insight into a conjecture of Erez concerning the Galois structure of the square root of the inverse different.
\end{abstract}

\maketitle

%\tableofcontents

\section{Introduction}
About thirty years ago, Erez initiated the study of the Hermitian-Galois module structure of the square root of the inverse different of finite Galois extensions of number fields.

To discuss this, we fix a finite Galois extension of number fields $L/K$, with $G = {\rm Gal}(L/K)$, and write $\mathcal{O}_L$ for the ring of algebraic integers of $L$ and $\mathcal{D}_{L/K}$ for the different of $L/K$. 

We note that there can exist at most one fractional ideal $\mathcal{A}_{L/K}$ of $L$ for which there is an equality of fractional ideals
\begin{equation}\label{def alk} (\mathcal{A}_{L/K})^2 = (\mathcal{D}_{L/K})^{-1}.\end{equation}
If such an ideal $\mathcal{A}_{L/K}$ exists, then (unsurprisingly) it is referred to  as the `square root of the inverse different' of $L/K$. More importantly, it will necessarily be both self-dual with respect to the canonical trace pairing $L \times L \to K$ of $L/K$ and also stable under the action of $G$ so that it can be considered as a module over the integral group ring $\ZZ[G]$.  

We henceforth assume that $L/K$ is `weakly ramified' in the sense of Erez \cite{E}. We recall that this means that the second ramification subgroup in $G$ (in the lower numbering) of every place of $K$ is trivial and that $L/K$ being tamely ramified is a sufficient, but not necessary, condition, for this to be satisfied. 

Suppose $G$ has odd order, we note that, Hilbert's classical formula for the valuation of $\mathcal{D}_{L/K}$ at each prime ideal implies that the fractional ideal $\mathcal{A}_{L/K}$ satisfying (\ref{def alk}) exists and, by a result of Erez in \cite{E}, $\mathcal{A}_{L/K}$ is a projective $\ZZ[G]$-module .

We also recall that the link between the Hermitian-Galois structures that are attached to $\mathcal{A}_{L/K}$ and the (second Adam-operator) `twisted' Galois-Gauss sums has been shown in special cases that $L/K$ is tamely ramified (in the article of Erez and Taylor \cite{ET}) and that $L/K$ is wildly ramified and satisfied a variety of restrictive hypotheses (in a series of articles of Vinatier \cite{V} and \cite{V2}, of Pickett and Vinatier \cite{PV} and of Pickett and Thomas \cite{PT}).

In particular, for {\em tamely} ramified, {\em odd} degree extensions $L/K$, Erez has shown in \cite{E} that the class $[\mathcal{A}_{L/K}]$ in the reduced projective class group ${\rm Cl}(\ZZ[G])$ defined by $\mathcal{A}_{L/K}$ is trivial (that is, equivalently, $\mathcal{A}_{L/K}$ is a free $\ZZ[G]$-module). 

Motivated by this result, and several other results (including extensive numerical computations), Vinatier has made the following conjecture (cf. \cite[Conj.]{V} and \cite[\S1.2]{CV}). 
\begin{conjecture}[{Vinatier}] \label{C:sv}
If $L/K$ is a weakly ramified Galois extension of number fields of odd degree, then $\mathcal{A}_{L/K}$ is a free $\ZZ[G]$-module. \end{conjecture}

\subsection{The conjecture of Erez}\label{S: erez conj}

We now assume that $L/K$ is {\em any} weakly ramified extension for which a fraction ideal $\mathcal{A}_{L/K}$ satisfying \eqref{def alk} exists. 

The first investigation of $\mathcal{A}_{L/K}$ in this case was undertaken by Caputo and Vinatier in the article \cite{CV}. By making a detailed study of certain torsion modules first considered by Chase \cite{Ch84}, they were able to prove that if $L/K$ is tamely ramified (so that the $\ZZ[G]$-modules $\mathcal{A}_{L/K}$ and $\mathcal{O}_L$ are both projective) and the decomposition group in $G$ of every ramified place of $L/K$ is abelian, then there is an equality in ${\rm Cl}(\ZZ [G])$
\begin{equation}\label{first one} [\mathcal{A}_{L/K}] = [\mathcal{O}_{L}]. \end{equation}

In this tamely ramified case one also knows, by Taylor's celebrated proof of Fr\"ohlich's Conjecture (in \cite{MT_81}), 
that there is an equality 
\begin{equation}\label{second one} [\mathcal{O}_L] = W_{L/K},\end{equation}
where $W_{L/K}$ denotes the Cassou-Nogu\`es-Fr\"ohlich root number class, which is defined in terms of Artin root numbers attached to non-trivial irreducible symplectic characters of $G$. 

By comparing the equalities (\ref{first one}) and (\ref{second one}), Caputo and Vinatier were, in particular, able to deduce the existence of a tamely ramified extension $L/K$ in which the $\ZZ[G]$-module $\mathcal{A}_{L/K}$ is not free, thereby showing (in view of Conjecture \ref{C:sv}) that the general theory of $\mathcal{A}_{L/K}$ can have features that do not seem apparent in the case of extensions of odd degree.  

Next we recall that, independently of any ramification hypotheses, Chinburg \cite{Chin_85} has defined
a canonical `$\Omega(2)$-invariant' $\Omega(L/K, 2)$ in $\Cl(\Z G)$, has shown that $\Omega(L/K, 2) = [\mathcal{O}_L]$ whenever $L/K$ is tamely ramified and has conjectured, as a natural generalization of the equality (\ref{second one}) proved by Taylor, that in all cases one should have  
\begin{equation}\label{third one}\Omega(L/K, 2) \stackrel{?}{=} W_{L/K}. \end{equation}

This conjectural equality has by now been extensively studied in the literature and is also known to have a natural interpretation as a consequence of the compatibility with respect to the functional equation of Artin $L$-functions of an important special case of the equivariant Tamagawa number conjecture.   

More concretely, following the results in \cite{AC} and a comparison of the equalities (\ref{first one}), (\ref{second one}) and (\ref{third one}), we consider a version of the conjecture made by Erez (cf. \cite[Ques. 2]{CV}  and \cite[Th. 1.5]{AC}).

\begin{conjecture} \label{C:boas} Let $L/K$ be a weakly ramified Galois extension of number fields for which $\mathcal{A}_{L/K}$ exists, and set $G = \Gal(L/K)$. Then, in ${\rm Cl}(\ZZ[G])$, one has 
\[
\Omega(L/K, 2) = [\mathcal{A}_{L/K}]+\mathcal{J}_{2, S, L/K}.
\]
\end{conjecture}
Here $\mathcal{J}_{2, S, L/K}$ is an element defined solely on the sign of the second Adams-operator twisted modified Galois-Jacobi sums attached to irreducible symplectic characters. In particular, the class in ${\rm Cl}(\ZZ[G])$ defined by $\mathcal{J}_{2, S, L/K}$ has order at most $2$ and is not always trivial in the class group (see \S \ref{S: Symp J}). 

This conjecture includes Conjecture \ref{C:sv} as a special case (since, if $|G|$ is odd, and so has no non-trivial irreducible symplectic characters, then one has $W_{L/K} = \mathcal{J}_{2, S, L/K}= 0$ ), and is perhaps the central question in the theory  of $\mathcal{A}_{L/K}$ for weakly ramified extensions of arbitrary degree. 

\subsection{The conjecture of Bley, Burns and Hahn}

Taking motivation from a somewhat different direction, Bley, Burns and Hahn have in \cite{BBH} recently introduced techniques of relative algebraic $K$-theory to formulate a precise conjectural link between $\mathcal{A}_{L/K}$ and twisted Galois-Gauss sums, and have thereby explained how much of the theory developed by Erez, by Erez and Taylor and by Vinatier can be refined.

More precisely, suppose $L/K$ is of odd degree, Bley, Burns and Hahn define a canonical relative element $\mathfrak{a}_{L/K}$ in the relative $K_0$-group $K_0(\mathbb{Z}[G],\mathbb{Q}^c[G])$ that measures the difference of the invariant arising from $\mathcal{A}_{L/K}$ and the second Adams-operator twisted Galois-Gauss sums of weakly ramified characters. 
They are able to prove in \cite[Th. 5.2]{BBH} that the element $\mathfrak{a}_{L/K}$ belongs to the torsion subgroup $K_0(\mathbb{Z}[G],\mathbb{Q}[G])_{\rm tor}$ of the subgroup $K_0(\mathbb{Z}[G],\mathbb{Q}[G]) \subset K_0(\mathbb{Z}[G],\mathbb{Q}^c[G])$, and that $\mathfrak{a}_{L/K}$ has good functorial properties under change of extension $L/K$. 

Motivated by these facts and extensive numerical computation, the authors are led to formulate the following conjecture.
\begin{conjecture}[{\cite[Conj. 10.7]{BBH}}]\label{bbh conj} If $L/K$ is any weakly ramified Galois extension of number fields of odd degree. Then, in $K_0(\ZZ[G], \QQ[G])$, one has $ \mathfrak{a}_{L/K} = \mathfrak{c}_{L/K}$.
\end{conjecture}
Here $\mathfrak{c}_{L/K}$ is a canonical `idelic twisted unramified characteristic' element such that it also belongs to $K_0(\mathbb{Z}[G],\mathbb{Q}[G])_{\rm tor}$ and enjoys the same functoriality properties under change of extension as does $\mathfrak{a}_{L/K}$. Recently the present author has provided new evidence \cite{Ku} for this conjecture in the setting of extensions of odd prime-degree.

\subsection{Main results}

We find it tempting, therefore, to wonder whether there are any useful links between Conjecture~\ref{C:boas} and the general approach of Bley, Burns and Hahn? 

The first thing to note in this regard is that all of the constructions in \cite{BBH} are made under the assumption that $G$ has odd order and so, in order to explore possible links to Conjecture \ref{C:boas}, it is necessary to define analogues of the relative elements $\mathfrak{a}_{L/K}$ and $\mathfrak{c}_{L/K}$ under the weaker hypothesis that $\mathcal{A}_{L/K}$ exists (rather than that $G$ has odd order). 

Fortunately, this is not very difficult - for the details see \S\ref{S: elements of BBH} and \S\ref{tiuc section}. In addition, we are also able to use the methods developed in \cite{BBH} to show that, in this more general setting, the elements $\mathfrak{a}_{L/K}$ and $\mathfrak{c}_{L/K}$ satisfy many of the same key properties as are established (for extensions of odd degree) in loc. cit. 

However, a much more serious issue that arises in this setting is that the (generalized) element $\mathfrak{a}_{L/K}$ involves Galois-Gauss sums attached to virtual characters that are obtained by taking the image of Artin characters under the second Adams operator $\psi_2$. 

The point here is that, whilst the approach of \cite{BBH} is rooted in the methods of Fr\"ohlich and Taylor and hence makes systematic use of functorial behaviour under passage to subgroups (or equivalently, via the appropriate `Hom-description', under induction of characters from subgroups), if the order of $G$ is even, then the operator $\psi_2$ does not always commute with induction of characters from subgroups. 

For this reason, the explicit computation of the Artin root numbers of (irreducible symplectic) characters twisted by $\psi_2$ can be considerably more difficult for extensions of even degree than for extensions of odd degree. 

In some cases, we are able to overcome these difficulties by using a detailed analysis of the root numbers of $\psi_2$-twisted characters that has recently been made in joint work \cite{AC} of ours with Agboola, Burns and Caputo. This result relies principally both on a purely representation-theoretic analysis of the extent to which $\psi_2$ commutes with induction functors and on the explicit computation of the root numbers of dihedral characters that is given by Fr\"ohlich and Queyrut in \cite{FQ73}. 

In particular, by combining these results with the general approach of Erez in \cite{E} and with the results of Bley and Cobbe in \cite{BC} (which themselves rely heavily on those of \cite{PV}), we are able to prove the following results. 

\begin{thm}\label{last intro thm} Let $L/K$ be a finite weakly ramified Galois extension of number fields for which $\mathcal{A}_{L/K}$ exists and set $G := \Gal(L/K)$. Then the following claims are valid. 
\begin{itemize}
\item[(i)] One has $\mathfrak{a}_{L/K} = \mathfrak{c}_{L/K}$ in $K_0(\ZZ[G],\QQ^c[G])$ provided that every place $v$ of $K$ that is wildly ramified in $L$ has the following three properties and no such $v$ is $2$-adic:
\begin{itemize}
\item[(a)] the decomposition subgroup in $G$ of any place of $L$ above $v$ is abelian;
\item[(b)] the inertia subgroup in $G$ of any place of $L$ above $v$ is cyclic;
\item[(c)] the completion of $K$ at $v$ is absolutely unramified.
\end{itemize}
\item[(ii)] Conjecture \ref{C:boas} is valid if the conditions of claim (i) are satisfied and, in addition, the inertia degree of each wildly ramified place $v$ in $L/K$ is prime to the absolute degree of the completion $K_v$. 
\end{itemize}
\end{thm}

We remark that claim (ii) of this result removes the locally abelian hypothesis from the main result of Caputo and Vinatier in \cite{CV}. 

We also note that, whilst the result for tamely ramified cases of claim (ii) was first proved in the jointly authored article \cite{AC}, the argument that we present here closely follows the classical methods of Erez in \cite{E} (see \S\ref{S: weakly classgroup}) rather than the methods of Agboola and McCulloh in \cite{AM} and so is different in key respects from that given in \cite{AC}. 

%We further note that it seems possible that the 
%technical condition on unramified characteristics that occurs in claim (ii) is not always satisfied and hence that this result could actually lead to a counterexample to the conjecture of Erez (for more details see Remark 
%\ref{possible counterexample remark}). 

Finally, we note that the result of claim (i) of Theorem \ref{last intro thm} makes it natural to wonder whether, with our generalized definitions of the elements $\mathfrak{a}_{L/K}$ and $\mathfrak{c}_{L/K}$, the equality $\mathfrak{a}_{L/K} = \mathfrak{c}_{L/K}$ that is conjectured by Bley, Burns and Hahn should be valid for \emph{all} weakly ramified Galois extensions $L/K$ for which the fractional ideal $\mathcal{A}_{L/K}$ exists. 

This is a question that we feel deserves further careful attention. 

However, since we have no evidence for an affirmative answer to it that goes beyond that in Theorem \ref{last intro thm} and, in particular, we lack the sort of numerical evidence that is provided in \cite{BBH}, we shall end this introduction by simply posing the following question. 

\begin{question} Let $L/K$ be a finite weakly ramified Galois extension of number fields for which $\mathcal{A}_{L/K}$ exists and set $G := \Gal(L/K)$. Then, in $K_0(\ZZ[G],\QQ^c[G])$, does one always have
	\[ \mathfrak{a}_{L/K} = \mathfrak{c}_{L/K}?\]  
\end{question}

\vskip 0.2truein \noindent{}{\bf Acknowledgements} I would like to thank David Burns for many helpful discussions and for insightful suggestions. I am very grateful to Adebisi Agboola, Luca Caputo and St\'ephane Vinatier for stimulating discussions concerning related works. 
\vskip 0.2truein

\section{Preliminary}

\subsection{Notations}\label{S: notation}
For a unital ring $A$, we write $A^\times$ for the multiplicative group of invertible elements of $A$ and $\zeta(A)$ for the centre of $A$. By an $A$-module, we shall always mean a left $A$-module. For each pair of $A$-modules $M$ and $M'$ we write ${\rm Is}_{A}(M, M')$ for the set of $A$-module isomorphisms from $M$ to $M'$ and ${\rm Aut}_{A}(M)$ for the group of $A$-module automorphisms of $M$. We also let $\mathcal{P}(A)$ denote the category of finitely generated projective $A$-modules.

Fix a finite group $\Gamma$, and a Dedekind domain $R$ of characteristic zero, with field of fractions $F$, and $E$ is an extension field of $F$. We let $\mathcal{O}_{E}$ denote the ring of integers of a field $E$. For any $R[\Gamma]$-module $M$, we write $M_{v}:=R_{v}\otimes_{R}M$, with $R_{v}$ the completion of $R$ at $v$, where $v$ is a (non-zero) prime ideal of $R$. 

Let $E/F$ be a finite Galois extension of fields. We let ${\rm Gal}(E/F)$ denote the Galois group of $E/F$, and we write the action of ${\rm Gal}(E/F)$ on $E$ by $x \mapsto g(x)$ for $x \in E$ and $g\in {\rm Gal}(E/F)$. We fix a separable closure $F^c$ of $F$, and write $\Omega_{F}$ for the absolute Galois group ${\rm Gal}(F^c/F)$ of $F$. 
For convenience, we take $\mathbb{Q}^c$ to be the algebraic closure of $\mathbb{Q}$ in $\mathbb{C}$. 

Throughout this article, we say $F$ is a `number field' if it is a finite extension of $\mathbb{Q}$ contained in $\mathbb{Q}^c$ and that $F$ is a `local field' if it is a finite extension of $\mathbb{Q}_v$ for some place $v$.

For a finite $\Gamma$, we write $\widehat{\Gamma}$ (resp. $\widehat{\Gamma}_\ell$) for the set of $\QQ^c$-valued (resp. $\QQ^c_\ell$-valued) irreducible characters of $\Gamma$ and $R_\Gamma$ (resp. $R_{\Gamma, \ell}$) for the additive group generated by $\widehat{\Gamma}$ (resp. $\widehat{\Gamma}_\ell$).

\subsection{Algebraic $K$-theory}
We fix a finite group $\Gamma$ and a Dedekind domain $R$ of characteristic zero with $F: = {\rm Frac}(R)$, and $E$ is a finite field extension of $F$. 

\subsubsection{The long exact sequence of relative $K$-theory}

%\cite[p. 237, (49.12)]{CR}
We write $K_{0}(R[\Gamma], E[\Gamma])$ for the algebraic $K_0$-group of the ring inclusion $R[\Gamma] \xrightarrow{\subset} E[\Gamma]$ described in \cite[p. 215]{swan} and we recall that it is generated by the symbols $[P, \phi, Q]$ where $P$, $Q \in \mathcal{P}(R[\Gamma])$ and $\phi \in {\rm Is}_{E[\Gamma]}(P_{E} , Q_{E})$. We also recall that there is a canonical decomposition
\begin{equation}\label{Eq: KT iso K0}
K_{0}(R[\Gamma], F[\Gamma])\cong \bigoplus_{v}K_{0}(R_{v}[\Gamma], F_{v}[\Gamma]), 
\end{equation}
where $v$ runs over all non-Archimedean places of $F$. This isomorphism is induced by the diagonal localisation homomorphism $(\pi_{\Gamma, v})_v$ (cf. discussion below \cite[(49.12)]{CR2}), where for each non-zero prime ideal $v$ of $R$, we write 
\begin{equation}\label{K_0 proj K_0_p}
\pi_{\Gamma, v}:  K_{0}(R[\Gamma], F[\Gamma]) \rightarrow K_{0}(R_{v}[\Gamma], F_{v}[\Gamma])
\end{equation}
for the homomorphism that sends the class of $[P, \phi, Q]$ to the class of $[P_{v}, F_{v}\otimes_{F}\phi, Q_{v}]$.

We also write $K_1(R[\Gamma])$ (resp. $K_1(E[\Gamma])$) for the Whitehead group of the ring $R[\Gamma]$ (resp. $E[\Gamma]$) and ${\rm Cl}(\ZZ[\Gamma])$ for the reduced projective class group of $R[\Gamma]$ (as in \cite[\S49A]{CR2}). In particular, we note that the reduced projective classgroup of $R[\Gamma]$ is isomorphic to the locally free classgroup of $R[\Gamma]$ (see Remark (49.11)(iv) of loc.cit., we therefore do not distinguish the notation of these two groups for the case of $R[\Gamma]$), and the latter group plays an important part in the `classical' Galois module theory (see \cite[Chap. I, \S2]{F83} for more details).

Then there exists a commutative diagram (taken from \cite[Th. 15.5]{swan})
\begin{equation} \label{Diagram: Kseq}
	\begin{CD} K_1(R[\Gamma]) @> >> K_1(E[\Gamma]) @> \partial^1_{R,E,\Gamma} >> K_0(R[\Gamma],E[\Gamma]) @> \partial^0_{R,E,\Gamma} >> {\rm Cl}(R[\Gamma])\\
		@\vert @A\iota_1 AA @A\iota_2 AA @\vert\\
		K_1(R[\Gamma]) @> >> K_1(F[\Gamma]) @> \partial^1_{R,F,\Gamma}  >> K_0(R[\Gamma],F[\Gamma]) @> \partial^0_{R,F,\Gamma}  >> {\rm Cl}(R[\Gamma]).
	\end{CD}
\end{equation}
in which the rows are the respective long exact sequences of relative $K$-theory (with the morphisms in the second row being completely analogous to those in the first): the homomorphism $\partial^{1}_{R,E,\Gamma}$ sends each pair $[E[\Gamma]^{n}, \phi]$, with $\phi$ in $\mathrm{Aut}_{E[\Gamma]}(E[\Gamma]^{n})$, to the class of $[E[\Gamma]^{n}, \phi, E[\Gamma]^{n}]$; for each pair $P$ and $Q$ in $\mathcal{P}(R[\Gamma])$ and each isomorphism of $E[\Gamma]$-modules $\phi: P_{E}\cong Q_{E}$, the homomorphism $\partial^{0}_{R,E,\Gamma}$ sends the class of $[P, \phi , Q]$ to the difference $[P]-[Q]$. For the vertical maps, $\iota_{1}$ and $\iota_{2}$ are the natural scalar extension morphisms (these maps are injective and will usually be regarded as inclusions).

\subsubsection{The reduced norm map}
Suppose $E$ is either a number field or a $p$-adic field (for some prime $p$), the `reduced norm' map discussed by Curtis and Reiner in \cite[45A]{CR2} induces an injective homomorphism of abelian groups 
\[ \mathrm{Nrd}_{E[\Gamma]}: K_{1}(E[\Gamma])\to \zeta (E[\Gamma]) ^{\times}. \]
We then define a subgroup of $\zeta(E[\Gamma])^\times$ by setting  
\[ \zeta (E[\Gamma]) ^{\times+} := {\rm im}(\mathrm{Nrd}_{E[\Gamma]}),\]
and we make much use of the following facts about this group.% \bigoplus_{i}\zeta(A_i)^{\times+}$. 
%\cite[(45.9)]{CR}, p.141(157)

\begin{lemma}[{The Hasse-Schilling-Maass Norm Theorem}]\label{Prop: Nrd bijective cond}\
	\begin{itemize}
		\item[(i)] $\zeta (E[\Gamma]) ^{\times+}$ contains $(\zeta (E[\Gamma]) ^{\times})^2$. 
		
		\item[(ii)] One has $\zeta (E[\Gamma]) ^{\times+} = \zeta (E[\Gamma]) ^{\times}$ in each of the following cases.
		\begin{itemize}
			\item[(a)] $E$ is algebraically closed;
			\item[(b)] $E$ is a number field and $\Gamma$ has no irreducible symplectic characters;
			\item[(c)] $E$ is $p$-adic (for any prime number $p$). 
		\end{itemize}
		\item[(iii)] If $E=\QQ$, then $\zeta(E[\Gamma])^{\times +} = \zeta(E[\Gamma])^{\times} \cap \zeta(\RR[\Gamma])^{\times+}$.
	\end{itemize}
\end{lemma}

\subsubsection{The extended boundary map of Burns and Flach}\label{S: extended boundary map}

The following construction (introduced by Burns and Flach in \cite[\S4.2]{BF}) is key to the formulation of arithmetic conjectures in relative $K$-groups. 

\begin{lemma}\label{ext bound hom} There exists a canonical homomorphism 
	\[ \delta_{\Gamma}: \zeta(\mathbb{Q}[\Gamma])^\times \to K_0(\mathbb{Z}[\Gamma],\mathbb{Q}[\Gamma])\]
	of abelian groups that has all of the following properties. 
	
	\begin{itemize} 
		\item[(i)] The connecting homomorphism $\partial^1_{\ZZ,\QQ,\Gamma}$ in (\ref{Diagram: Kseq}) is equal to the composite $\delta_{\Gamma}\circ {\rm Nrd}_{\QQ[\Gamma]}$.
		\item[(ii)] If $x$ belongs to $\zeta(\QQ[\Gamma])^{\times +}$, then one has $\delta_\Gamma(x) = \partial^1_{\ZZ,\QQ,\Gamma}\bigl(({\rm Nrd}_{\QQ[\Gamma]})^{-1}(x)\bigr)$. In particular, if ${\rm Nrd}_{\QQ[\Gamma]}$ is bijective, then $\delta_\Gamma = \partial^1_{\ZZ,\QQ,\Gamma}\circ ({\rm Nrd}_{\QQ[\Gamma]})^{-1}$. 
		\item[(iii)] Fix a prime $\ell$ and write $j_{\ell, *}$ for the natural projection map $K_0(\ZZ[\Gamma],\QQ[\Gamma]) \to K_0(\ZZ_\ell[\Gamma],\QQ_\ell[\Gamma])$ (as in \eqref{K_0 proj K_0_p}). Then, regarding $\zeta(\QQ[\Gamma])^\times$ as (an obvious) subgroup of $\zeta(\QQ_\ell[\Gamma])^\times$, one has 
		\[ j_{\ell, *}\circ \delta_\Gamma = \partial^1_{\ZZ_\ell,\QQ_\ell,\Gamma}\circ ({\rm Nrd}_{\QQ_\ell[\Gamma]})^{-1}.\]  
	\end{itemize}
\end{lemma} 

\begin{remark}\label{Rem: extend bound extend to R} The homomorphism $\delta_\Gamma$ in Lemma \ref{ext bound hom} extends to give a group homomorphism $\zeta(\mathbb{R}[\Gamma])^\times \to K_0(\mathbb{Z}[\Gamma],\mathbb{R}[\Gamma])$ (this was the original construction of Burns and Flach) but not, in general, to a group homomorphism 
	$\zeta(\mathbb{C}[\Gamma])^\times \to K_0(\mathbb{Z}[\Gamma],\mathbb{C}[\Gamma])$ (see, for example, Breuning \cite[Prop. 2.11]{B_phd}).\end{remark}

\begin{remark}\label{Rem: extend bound = red norm} Following the explicit recipe given in \cite[\S4.2, Lem. 9]{BF} 
for each $x$ in $\zeta (\QQ[\Gamma])^\times$, there is an equality
	\[\delta_\Gamma(x) = \partial^1_{\ZZ,\QQ,\Gamma}\bigl(({\rm Nrd}_{\QQ[\Gamma]})^{-1}(x^2)\bigr) - \sum_\ell \partial^1_{\ZZ_\ell,\QQ_\ell,\Gamma}\bigl(({\rm Nrd}_{\QQ_\ell[\Gamma]})^{-1}(x_\ell)\bigr).\]
	where in the sum $\ell$ runs over all primes, $x_\ell$ denotes $x$ regarded as an element of $\zeta(\QQ_\ell[\Gamma])^\times = \zeta(\QQ_\ell[\Gamma])^{\times+}$. Here we use the fact that $x^2$ belongs to $\zeta(\QQ[\Gamma])^{\times +}$ (by Lemma \ref{Prop: Nrd bijective cond}(i)),	and regard each group $K_0(\ZZ_\ell[\Gamma],\QQ_\ell[\Gamma])$ as a subgroup of 
	$K_0(\ZZ[\Gamma],\QQ[\Gamma])$ by the canonical decomposition (\ref{Eq: KT iso K0}). 
\end{remark}

\subsubsection{Parametrizing central elements}

Let $E$ be an algebraically closed field of characteristic $0$. We fix a finite group $\Gamma$, and write $\widehat{\Gamma}:=\widehat{\Gamma}(E)$ for the set of $E$-valued irreducible characters of $\Gamma$ and $R_\Gamma$ for the additive group generated by $\widehat{\Gamma}$. For each $\chi$ in $\widehat{\Gamma}$, we obtain a primitive idempotent of the centre $\zeta(E[\Gamma])$ of $E[\Gamma]$ by setting $e_\chi =\frac{\chi(1)}{|\Gamma|} \sum_{g\in \Gamma}\chi(g^{-1})g$. In addition, the standard orthogonality relations for irreducible characters implies that the set $\{e_\chi\}_{\chi \in \widehat{\Gamma}}$ is an $E$-basis of $\zeta(E[\Gamma])$ (see, for example, \cite[Chap. 6.3, Exer. 6.4]{S1}). 

It follows that each element of $\zeta(E[\Gamma])$, respectively $\zeta(E[\Gamma])^\times$, can be written uniquely in the form
\begin{equation}\label{Eq: element in centre}
	x= \sum_{\chi\in \widehat{\Gamma}}e_{\chi}\cdot x_{\chi} , \quad \text{ with $x_\chi\in E$, respectively 
		$x_\chi \in E^\times$, for all $\chi$}.
\end{equation}

In later arguments, we will use the following two consequences of this decomposition. Firstly, the horizontal isomorphism in \cite[Chap. II, \S1, Lem. 1.6]{F83} implies that, for every $u$ in $\mathrm{GL}_{n}(E[\Gamma])$ and every $\chi$ in $\widehat{\Gamma}$, one has ${\rm Det}(u)_\chi = {\rm Det}_\chi (u) = {\rm Nrd}_{E[\Gamma]}(u)_\chi \in E$. Secondly, the proof of Lemma~\ref{Prop: Nrd bijective cond} implies that ${\rm im}(\mathrm{Nrd}_{\RR[\Gamma]}) = \zeta (\RR[\Gamma])^{\times+} $ is equal to the subgroup of $\zeta(\CC[\Gamma])^{\times}$ that is defined by the following explicit conditions on the individual coefficients $x_\chi$: 
\begin{equation}\label{description of image nrd RR}
\left\{
\sum_{\chi\in \widehat{\Gamma}}e_{\chi}\cdot x_{\chi}\ \middle\vert \begin{array}{l}
x_{\bar{\chi}} = \overline{x_\chi} \; \text{  for all  } \; \chi \in \widehat{\Gamma}(\CC); \\
\text{ and } x_{\chi}>0 \; \text{  for all  } \; \chi\in {\rm Symp}(\Gamma),
\end{array}\right\},
\end{equation}
where $\overline{x_\chi}$ denotes the complex conjugation of $x_\chi$ and ${\rm Symp}(\Gamma)$ denotes the set of irreducible {\em complex} symplectic characters of $\Gamma$.

\subsubsection{Induction functors}\label{S: K_0 functors}
We let $R$ denote either $\ZZ$ or $\ZZ_\ell$ for a prime number $\ell$, and $F$ denotes the corresponding fields $\QQ$ or $\QQ_\ell$. Suppose $E$ is a finite field extension of $F$ contained in $F^c$. Let $G$ be a finite group, and let $J$ be a subgroup of $G$, we let ${\rm res}^G_J$ denote the restriction homomorphism $R_G \rightarrow R_J$.

We write ${\rm i}^{G}_J$ for the induction functor of $\mathcal{P}(R[J]) \rightarrow \mathcal{P}(R[G])$ via applying $R[G] \otimes_{R[J]}$, and similarly for the $E[G]$-module, noting that it preserves isomorphisms and short exact sequences. This functor induces a homomorphism of relative $K$-groups that we denote by 
\begin{align}\label{Eq: K_0 ind def}
{\rm i}^{G, *}_{J, E} : K_{0}(R[J], E[J])   \, & \rightarrow  K_{0}(R[G], E[G]),\\
[P, \phi, Q] \, & \mapsto [{\rm i}^{G}_J P, {\rm i}^{G}_J \phi, {\rm i}^{G}_J Q].\nonumber 
\end{align}

We also define a map $\mathrm{\tilde{i}}_{J}^{G}: \zeta(F^c[J])^{\times} \rightarrow \zeta(F^c[G])^{\times}$ by setting, for each $x \in \zeta(F^c[J])^{\times} $ and $\chi \in \widehat{G}$ (in terms of the decomposition in \eqref{Eq: element in centre}),
\begin{equation}\label{Eq: center ind def}
	\mathrm{\tilde{i}}_{J}^{G} (x)_{\chi}= \prod_{\varphi \in \widehat{J}} x_{\varphi}^{<\mathrm{res}^{G}_{J} \chi , \varphi >_{J}}. 
\end{equation} 

Set $\delta_{R, E, \Gamma}:= \partial^{1}_{R, E, \Gamma}\circ ({\rm Nrd}_{E[\Gamma]})^{-1} : \zeta(E[\Gamma])^{\times+} \rightarrow K_0(R[\Gamma], E[\Gamma])$, we recall from \cite[p. 581]{BB} that
\begin{equation}\label{Diagram: K_0 ind commute}
\mathrm{i}_{J, E}^{G, *} \circ \delta_{R, E, J}  =   \delta_{R, E, G} \circ \mathrm{\tilde{i}}_{J}^{G} .
%
%\begin{CD} 
%\zeta(E[J])^{\times+} @> \delta_{R, E, J} >> K_{0}(R[J],E[J])\\
%
%@V \mathrm{\tilde{i}}_{J}^{G} VV @VV \mathrm{i}_{J, E}^{G, *} V\\
%
%\zeta(E[G])^{\times+} @> \delta_{R, E, G} >> K_{0}(R[G], E[G]).
%\end{CD}
\end{equation}

\subsubsection{Taylor's Fixed Point Theorem}\label{S: taylor fixed point}
To finish this section, we shall recall an important result of Taylor. Fix a finite group $\Gamma$ and a prime number $\ell$ and write $\mathbb{Q}_{\ell}^{t}$ for the maximal tamely ramified extension of $\mathbb{Q}_{\ell}$ in $\mathbb{Q}_{\ell}^{c}$. We also write $\mathcal{O}^{t}_{\ell}$ for the valuation ring of $\mathbb{Q}_{\ell}^{t}$ and consider the associated homomorphism of relative $K$-groups.
\begin{align}\label{Eq: Taylor map def0}
j^{t}_{\ell, *} : K_{0}(\mathbb{Z}_{\ell}[\Gamma], \mathbb{Q}^{c}_{\ell}[\Gamma]) \, & \to K_{0}(\mathcal{O}_{\ell}^{t}[\Gamma], \mathbb{Q}^{c}_{\ell}[\Gamma]),\\
[P, \phi , Q] \, & \mapsto [\mathcal{O}_{\ell}^{t}\otimes_{\mathbb{Z}_{\ell}}P , \phi,  \mathcal{O}_{\ell}^{t}\otimes_{\mathbb{Z}_{\ell}} Q].\nonumber
\end{align}

Then, by Taylor's Fixed Point Theorem \cite[Chap. 8, \S1]{M84}, %(see also \cite[Chap. II, \S6]{F83}), 
one knows that the composite homomorphism $K_{0}(\mathbb{Z}_{\ell}[\Gamma], \mathbb{Q}_{\ell}[\Gamma]) \rightarrow K_{0}(\mathbb{Z}_{\ell}[\Gamma], \mathbb{Q}^{c}_{\ell}[\Gamma]) \xrightarrow{j^{t}_{\ell, *}} K_{0}(\mathcal{O}_{\ell}^{t}[\Gamma], \mathbb{Q}^{c}_{\ell}[\Gamma])$, is injective, where the first map is the inclusion induced by the relevant case of the diagram \eqref{Diagram: Kseq}. 

We will also use the fact (taken from \cite[(49)]{BB}) that the kernel of the composite homomorphism 
\begin{equation}\label{Eq: Taylor kernel Det} \zeta(\QQ_\ell^t[\Gamma])^\times 
	\xrightarrow{{\rm Nrd}_{\mathbb{Q}^t_\ell[\Gamma]}^{-1}} K_1(\QQ_\ell^t[\Gamma]) 
	\xrightarrow{\partial^1_{\mathcal{O}_\ell^t,\QQ_\ell^t,\Gamma}} K_0(\mathcal{O}_\ell^t[\Gamma],\QQ_\ell^t[\Gamma])\end{equation}
is equal to ${\rm Det}(\mathcal{O}_{\ell}^t[\Gamma]^\times)$, where ${\rm Det}$ denotes the generalized determinant map discussed in \cite[Chap. I, \S2 and Chap. II, \S2]{F83}.

\subsection{Adams operations and Galois-Jacobi sums}
In this final section, we recall the definitions of  Adams operations on the ring of virtual characters, modified Galois-Gauss sums and the Galois-Jacobi sums considered in \cite[\S4A]{BBH}.

\subsubsection{Adams operations} 
We fix a finite group $\Gamma$. For each natural number $k$ and each $\chi\in R_{\Gamma}$, we define a class function $\psi_k(\chi)$ on $\Gamma$ by setting $\psi_{k}(\chi)(\gamma)=\chi(\gamma^k)$ for all $\gamma \in \Gamma$. The function $\psi_k$ from $R_\Gamma$ to the set of class functions on $\Gamma$ is called the $k$-th Adams operator for $\Gamma$ (cf. \cite[Prop. (12.8)]{CR1}). The following general results about Adams operators will be very useful in the sequel (in the case of $k = 2$). 
\begin{prop}\label{Prop: adam}\
	\begin{itemize}
		\item[(i)]	$\psi_{k}$ is an endomorphism of $R_\Gamma$. In particular, for each $\chi$ in $R_\Gamma$, one has $\psi_k(\chi) \in R_\Gamma$.  
		\item[(ii)] $\psi_{k}$ commutes with restriction and inflation functors, and with the action of $\Omega_{\mathbb{Q}}$ on $R_{\Gamma}$.
		\item[(iii)] For each $\chi\in R_{\Gamma}$, one has $\det_{\psi_{k}(\chi)}=(\det_{\chi})^{k}.$
		\item[(iv)] If $k$ is prime to the order of $\Gamma$, then  $\psi_{k}$ commutes with induction functors, and a character $\chi$ is irreducible if and only if $\psi_{k}(\chi)$ is irreducible. 
	\end{itemize}
\end{prop}
\begin{proof} 
	Claim (i) is stated in \cite[Cor. (12.10)]{CR1} and claim (ii) and (iii) are taken from \cite[Prop.-Def. 3.5]{E}. The assertions in claim (iv) follow from  \cite[p.15]{Ke} and \cite[\S9.1, Exer. 9.4]{S1} respectively.
\end{proof}

\begin{remark}\label{Rem: adam commute ind coprime}
	If $k$ is not prime to $|\Gamma|$, then $\psi_{k}$ does not in general commute with the induction functors or preserve irreducibility. It is also important to note that, if $|\Gamma|$ is even, then $\psi_{2}$ does not in general preserve the set of symplectic characters (for an explicit example, see \cite[Lem. 7.3(d)]{AC}). 
\end{remark}

\begin{remark}\label{endo def} For each pair of integers $m$ and $n$, and each natural number $k$, the canonical decomposition \eqref{Eq: element in centre} allows us to define an endomorphism $m + n\cdot \psi_{k,\ast}$ of the multiplicative  group $\zeta (\mathbb{Q}^{c}[\Gamma])^\times$ in the following way: for each element $x$ of $\zeta (\mathbb{Q}^{c}[\Gamma])^\times$, the element $(m+n\cdot \psi_{k,*})(x)$ is uniquely specified by the condition that $((m+n\cdot \psi_{k,*})(x))_{\chi}:= (x_{\chi})^{m}\cdot (x_{\psi_{k}(\chi)})^{n}$ for every irreducible character $\chi$ in $\widehat{\Gamma}$.  
\end{remark}

\subsubsection{The unramified characteristic}
We recall the definition of `unramified characters' from \cite[Chap. I, \S5, (5.6)]{F83}. Fix a finite Galois extension $E/F$ of non-Archimedean local fields and set $\Gamma:=\mathrm{Gal}(E/F)$, we write $\Gamma_0$ for the inertia subgroup of $\Gamma$.

\begin{definition}\label{Def: unrami char}\
\begin{itemize}
\item[(i)] An irreducible character $\chi\in \widehat{\Gamma}$ is said to be `unramified' if $\Gamma_0$ is contained in the kernel of $\chi$ (or, equivalently, if $\chi(\gamma) = \chi(1)$ for all $\gamma \in \Gamma_0$). If $\chi$ is not unramified, then it is said to be `ramified'.
\item[(ii)] Each $\chi \in R_\Gamma$ can be written uniquely as a finite sum $\chi=\sum_{i \in I} m_i \cdot \mu_i$, where each $m_i$ is an integer and each $\mu_i$ belongs to $\widehat{\Gamma}$. The `unramified part' of $\chi$ is the element of $R_\Gamma$ obtained by setting $n(\chi) := \sum_{i \in I} m_i \cdot n(\mu_i)$, with
\[n(\mu_i) :=  \left\{ \begin{array}{ll}
\mu_i, & \quad \mbox{ if } \mu_i \mbox{ is unramified};\\
0	, & \quad \mbox{ if }\mu_i \mbox{ is ramified}.\end{array} \right. \]
\end{itemize}
\end{definition}

We next recall the `unramified characteristic' that is defined in \cite[Chap. IV, \S1, (1.1)]{F83}.
\begin{definition}\label{Def: unrami char y} 
\begin{itemize}
\item[(i)] For each $\phi \in \widehat{\Gamma}$, the `unramified characteristic' of $\phi$ is defined by setting 
\[y(F, \phi) = \left\{ \begin{array}{ll}
	1, & \quad \mbox{ if }\phi \mbox{ is ramified};\\
	(-1)^{\phi(1)}\det_\phi(\sigma), & \quad \mbox{ if }\phi \mbox{ is unramified},\end{array} \right. \]
where $\sigma$ is the Frobenius element in $\Gamma/\Gamma_{0}$ lifted to $\Gamma$.
\item[(ii)] The `equivariant unramified characteristic' of $E/F$ is defined by $y_{E/F}:=\sum_{\chi \in \widehat{\Gamma}} e_\chi \cdot y(F, \chi)$.
\end{itemize}
\end{definition}

We then recall (from \cite[Chap. IV, \S1, Th. 29(i)]{F83}) that the function $\phi \mapsto y(F, \phi)$ lies in ${\rm Hom}^{+}(R_{\Gamma}, \mu(\mathbb{Q}^c))^{\Omega_{\mathbb{Q}}}$,
where $\mu(\mathbb{Q}^c)$ is the group of roots of unity in $\mathbb{Q}^c$ and so for all prime $\ell$,
\begin{equation}\label{unram maximal order} 
y_{E/F}\in \zeta(\mathbb{Q}[\Gamma])^\times \subseteq \zeta(\mathbb{Q}_\ell[\Gamma])^\times.\end{equation}

\subsubsection{Modified global  Galois-Gauss sums}\label{S: modified GGS}
Let $L/K$ be a finite Galois extension of number fields and set $G:=\mathrm{Gal}(L/K)$. In terms of \eqref{Eq: element in centre}, we define the following `equivariant' elements in $\zeta(\mathbb{Q}^{c}[G])^\times$.

\begin{definition}\label{Def: equiv global GGS}\
	
	\begin{itemize}
		\item[(i)] The `equivariant Galois-Gauss sum' of $L/K$, $\tau_{L/K}:=\sum_{\chi \in \widehat{G}} e_\chi \cdot \tau(K,\chi)$, where $\tau(K,\chi)$ is the global Galois-Gauss sums (cf. \cite[Chap. I, \S5, (5.22)]{F83}).
		\item[(ii)] The `global unramified characteristic' of a virtual character $\chi\in R_G$ is defined by setting $y(K, \chi) = \prod_{v|d_L } y(K_v, \chi_v)$, where $d_L$ denotes the (absolute) discriminant of $\mathcal{O}_L$.
		
		\item[(iii)] The `equivariant unramified characteristic' of $L/K$, $y_{L/K}:=\sum_{\chi \in \widehat{G}} e_\chi \cdot y(K, \chi) \in \zeta(\mathbb{Q}[G])^\times$.
		
\item[(iv)] The `modified equivariant Galois-Gauss sum' of $L/K$, $\tau'_{L/K}:=\tau_{L/K}\cdot y_{L/K}^{-1}$.

\item[(v)] The `absolute Galois-Gauss sum' of $L/K$,  $\tau^{\dagger}_{L/K}:=\sum_{\chi \in \widehat{G}} e_{\chi}\cdot \tau(\mathbb{Q}, \mathrm{ind}^{\mathbb{Q}}_{K}\chi)$.
\item[(vi)] The `induced discriminant' of $L/K$, $\tau^{G}_{K} := \sum_{\chi \in \widehat{G}} e_{\chi}\cdot \tau_K^{\chi(1)}$ with $\tau_K := \tau(\mathbb{Q}, \mathrm{ind}^{\mathbb{Q}}_{K} \mathbf{1}_{K})$.
\end{itemize}
\end{definition}

\subsubsection{Galois-Jacobi sums} 
Let $L/K$ be a finite Galois extension of number fields and set $G:=\mathrm{Gal}(L/K)$. 
In the following definition, we use a particular case of the endomorphisms defined in Remark \ref{endo def}. 

\begin{definition}
	The `second Galois-Jacobi sum' for $L/K$ is the element of $\zeta(\QQ^c[G])^\times$ obtained by setting $J_{2, L/K}:=(\psi_{2, *}-2)(\tau_{L/K})$.
\end{definition}

The following properties of these elements are important in the sequel.  

\begin{prop}\label{Prop: JG decomp}
	For each finite place $v$ of $K$, we fix some $w\in L$ above $v$ and identify ${\rm Gal}(L_w/K_v)$ with the decomposition subgroup $G_w$ of $w$ in $G$.	
	\begin{itemize}
		\item[(i)] $J_{2, L/K} =\prod_{v\in S_f(K)} \mathrm{\tilde{i}}_{G_{w}}^{G} (J_{2, L_{w}/K_{v}})$. 
		\item[(ii)] $(1- \psi_{2, *})(y_{L/K}) = \prod_{v|d_L} \mathrm{\tilde{i}}_{G_{w}}^{G} ((1- \psi_{2, *}) (y_{L_w/K_v}))$. %where $d_L$ denotes the discriminant of $\mathcal{O}_L$.
		\item[(iii)] $J_{2, L/K}\cdot (1-\psi_{2, *})(y_{L/K}) = \tau^{G}_{K}\cdot(\psi_{2, *}-1)(\tau_{L/K}^{\prime})(\tau^{\dagger}_{L/K})^{-1}$.
	\end{itemize}
\end{prop}
\begin{proof} 
To prove claim (i), by the decomposition results of global Galois Gauss sums (cf. \cite[Chap. III, \S2, Cor. to Th. 18]{F83}), one has $\tau(K, \psi_{2}(\chi))=\prod_{v\in S_f(K)} \tau(K_{v}, \psi_{2}(\chi)_{v})=\prod_{v\in S_f(K)} \tau(K_{v}, \psi_{2}(\chi_{v}))$,
where the second equality follows from Proposition~\ref{Prop: adam}(ii). Then, the result follows from the definition of $\mathrm{\tilde{i}}_{G_{w}}^{G}$ given in \eqref{Eq: center ind def}.
	
We can derive the result of claim (ii) via the same argument, replacing the decomposition result by the definition of the global unramified characteristic in Definition~\ref{Def: equiv global GGS}(ii).
	
Claim (iii) follows by an explicit comparison of the definitions (see \cite[(4.5)]{BBH}, noting that the assumption that $G$ has odd order is not necessary here).
\end{proof}

\begin{prop}\label{Prop: JG sums equivariant} $J_{2, L/K}$ belongs to $\zeta(\mathbb{Q}[G])^{\times}$. 
\end{prop}
\begin{proof}
The proof is exactly the same to \cite[Lem. 4.4]{BBH} (we note that the assumption that $(k, |G|)=1$ is not necessary here). 
\end{proof}

\begin{remark}\label{Rem: loc JG sums def} 
If $E/F$ is a finite Galois extension of local fields and set $\Gamma = \Gal(E/F)$, then one obtains a local `second Galois-Jacobi sum' by setting $J_{2,E/F} := (\psi_{2,\ast}-2)(\tau_{E/F})$ where $\tau_{E/F}$ is defined to be the local analogy of the equivariant Galois-Gauss sum defined above. The same arguments above show that these elements have the following properties: 
\begin{itemize}
\item[(i)] $J_{2, E/F}\in  \zeta(\mathbb{Q}[\Gamma])^{\times}$.
The proof of this result is exactly the same to the global case, with the Galois action of local Galois-Gauss sums (cf. \cite[p. 42, Th. 5.1]{M}). 
\item[(ii)] $J_{2,  E/F}\cdot (1 - \psi_{2, *})(y_{E/F}) = \tau^{\Gamma}_{F}\cdot(\psi_{2, *}-1)(\tau_{E/F}^{\prime})(\tau^{\dagger}_{E/F})^{-1}$, where  $\tau_{E/F}^{\prime}$ and $\tau^{\dagger}_{E/F}$ are defined to be the local analogy of the elements defined in Definition~\ref{Def: equiv global GGS}.
\end{itemize}
\end{remark}

\section{The square root of the inverse different and the relative $K$-theory}
In this section, we first recall basic properties of the square root $\mathcal{A}_{L/K}$ of the inverse different of a Galois extension $L/K$ that is weakly ramified in the sense of Erez \cite{E}. 

We then generalise the construction by Bley, Burns and Hahn of a canonical element in relative $K$-theory that compares $\mathcal{A}_{L/K}$ to Galois-Gauss sums. 

\subsection{Weakly ramified extensions}
If $E/F$ is a Galois extension of $p$-adic fields, with $\Gamma = \Gal(E/F)$, then for each non-negative integer $i$, we write $\Gamma_{i}$ for the $i$-th ramification subgroup (in the lower numbering) of $\Gamma$. In particular, we recall that $\Gamma_{0}$ is the inertia subgroup of $\Gamma$ and that $\Gamma_{1}$ is the Sylow $p$-subgroup of $\Gamma$.  

Throughout this section, we assume $L/K$ to be a finite Galois extension of number fields and $G := {\rm Gal}(L/K)$. We also assume that the square root of the inverse different of $L/K$ exists. If $v$ is a non-Archimedean place of $K$ and $w$ is a place of $L$ above $v$, then we always identify the decomposition subgroup $G_w$ of $w$ in $G$ with the local Galois group $\mathrm{Gal}(L_w/K_v)$. %Then, the following definition originates with Erez \cite[Def. 2.1]{E}.
%
%\begin{definition} A Galois extension $E/F$ of $p$-adic fields is \emph{weakly ramified} if the subgroup $\Gamma_2$ is trivial. A Galois extension $L/K$ of number fields is \emph{weakly ramified} if for every non-Archimedean place $w$ of $L$ above a place $v$ of $K$, the extension $L_w/K_v$ is weakly ramified (or equivalently, the group $G_{w, 2}$ is trivial).
%\end{definition}
%

The importance of weak ramification \cite[Def. 2.1]{E} for our theory is explained by the following result. 

\begin{prop}\label{Lemma: weakly rami proj iff}
	Let $L/K$ be a finite Galois extension of number fields for which the square root of the inverse different exists. Then, the following conditions are equivalent.
	\begin{itemize}
		\item[(i)] $\mathcal{A}_{L/K}$ is a projective $\mathcal{O}_K[G]$-module. 
		\item[(ii)] For every finite place $v$ of $K$, the $\mathcal{O}_{K_v}[G_w]$-module $\mathcal{A}_{L_{w}/K_v}$ is free.
		\item[(iii)] $L/K$ is weakly ramified.
	\end{itemize}
\end{prop}

\begin{remark} In \cite[p. 109, footnote 1]{CV}, Caputo and Vinatier first pointed out (without proof) that the original argument \cite[Th. 1]{E} of Erez for extensions of odd degree can be extended to show that if $L/K$ is weakly ramified and  $\mathcal{A}_{L/K}$ exists, then it must be a projective $\mathcal{O}_K[G]$-module. \end{remark}

The equivalence of the conditions (i) and (ii) follows from a general result of integral representation theory (see, for example, \cite[Th. (32.11)]{CR1}). 

To prove the equivalence of conditions (ii) and (iii), we fix a finite Galois extension $E/F$ of $p$-adic fields for which $\mathcal{A}_{E/F}$ exists, set $\Gamma = \Gal(E/F)$ and write $\mathfrak{p}_{E}$ for the maximal ideal of $\mathcal{O}_E$. It is sufficient for us to prove that $\mathcal{A}_{F/E}$ is a free 
$\mathcal{O}_E[\Gamma]$-module if and only if $E/F$ is weakly ramified. Our argument relies on the following result of K\"{o}ck. 

\begin{lemma}[{K\"{o}ck, \cite[Th. 1.1]{K}}]\label{Lem: Koeck}
	Fix an integer $n$. Then, the fractional $\mathcal{O}_E$-ideal $\mathfrak{p}_{E}^n$ is free over $\mathcal{O}_F[\Gamma]$ if and only if both $E/F$ is weakly ramified and $n\equiv 1 \pmod{\Gamma_{1}}$.
\end{lemma}

This result immediately implies that condition (ii) implies condition (iii) and also shows that conditions (iii) implies condition (ii) provided that, via defining $n$ by the equality $\mathcal{A}_{E/F} = \mathfrak{p}_E^n$, one has $n \equiv 1 \pmod{|\Gamma_1|}$.

To prove this, we recall (from \cite[Chap. IV, \S2, Prop. 4]{S2}) the Hilbert's formula for the valuation of the different $\mathcal{D}_{E/F}$ of $E/F$,
\begin{equation}\label{Eq: Hilbert}
	{\rm ord}_{E}(\mathcal{D}_{E/F})=\sum_{i=0}^{\infty}(|\Gamma_{i}|-1).
\end{equation}

We set $I:=\Gamma_0$, $W:=\Gamma_1$ and $C:=I/W$, and we recall from \cite[Chap. IV, \S2, Cor. 4 to Prop. 7]{S2} that $|C| = |I|/|W|$ is prime to $|W|$. Then, since $\Gamma_2$ is assumed to be trivial, Hilbert's formula implies that 
\begin{align}\label{explicit computation} 
n = {\rm ord}_{E}(\mathcal{A}_{E/F})= -\frac{{\rm ord}_{E}(\mathcal{D}_{E/F})}{2} =&\, -\frac{1}{2}(|I|-1+|W|-1 )\nonumber \\ 
=&\, -\frac{1}{2}(|C|\cdot|W|-1+|W|-1 )\nonumber \\
=& \, -\frac{|C|+1}{2}\cdot |W| +1.
\end{align}

If $p=2$, then $|W|$ is even and so $|C|$ is odd and, hence, $(|C| +1)/2$ is an integer. So, this formula implies the required congruence $n \equiv 1 \pmod{|W|}$. 

If $p > 2$, then $|W|$ is odd and hence, since $n$ is an integer, this formula implies that $|C|+1$ is even, and therefore, that $n \equiv 1 \pmod{|W|}$, as required. 

This completes the proof of Proposition~\ref{Lemma: weakly rami proj iff}.

\begin{remark}\label{Cor: odd inertia} If $E/F$ is a weakly ramified Galois extension of $p$-adic fields, with $\Gamma = \Gal(E/F)$,  then the order of the inertia subgroup $\Gamma_0$ is odd if either  
	\begin{itemize}
		\item[(i)] $E/F$ is tamely ramified and $\mathcal{A}_{E/F}$ exists, or 
		\item[(ii)] $p$ is odd  and $\mathcal{A}_{E/F}$ exists.
	\end{itemize}	
	The claim in case (i) follows directly from Hilbert's formula \eqref{Eq: Hilbert} and the fact that, in this case, $\Gamma_i$ is trivial for all $i > 0$. Case (ii) is valid since, $|\Gamma_1|$ is odd in this case and so the computation (\ref{explicit computation}) implies that $|C|=|I|/|\Gamma_1|$, and hence also $|I|$, are odd.
\end{remark}

\subsection{The canonical relative elements of Bley, Burns, and Hahn}\label{S: elements of BBH}
In this section, we extend the key definition of Bley, Burns, and Hahn from \cite{BBH} to the setting of all weakly ramified Galois extensions $L/K$ of number fields for which the inverse different is a square. We also show that this element decomposes as a sum of canonical elements arising from the extensions obtained by completing $L/K$ at the places that ramify in $L/K$. %al fields. %al analogue of this element e can also defined the canonical local element as in \S7 of loc. cit., and we therefore show that the global element can be decomposed as a sum of local elements.

\subsubsection{The global element}\label{def global element}
Let $L/K$ be a weakly ramified finite Galois extension of number fields with Galois group $G:={\rm Gal}(L/K)$. In this section, we define the canonical relative elements of $K_{0}(\mathbb{Z}[G], \mathbb{Q}^{c}[G])$ using the construction in \cite[\S2A3 and \S5]{BBH}.

To start, we identify the set $\Sigma(L)$ of field embeddings $L\rightarrow \mathbb{Q}^{c}$ with the set of field embeddings $L\rightarrow \mathbb{C}$ and consider the free $\ZZ[G]$-module $ H_{L}:=\prod_{\Sigma(L)}\mathbb{Z}$ (upon which $G$ acts via its natural composition action on $\Sigma(L)$). We then consider the isomorphism of $\QQ^c[G]$-modules $\kappa_{L}:\mathbb{Q}^{c}\otimes_{\mathbb{Q}}L \rightarrow \prod_{\Sigma(L)}\mathbb{Q}^{c}= \mathbb{Q}^{c}\otimes_{\mathbb{Z}}H_{L}$ that sends $z\otimes l$ for each $z \in \mathbb{Q}^c$ and $l \in L$ to the vector $(\sigma(l)z)_{\sigma\in \Sigma(L)}$. Then, for any full projective $\mathbb{Z}[G]$-sublattice $\mathcal{L}$ of $L$, we obtain an element of $K_{0}(\mathbb{Z}[G], \mathbb{Q}^{c}[G])$ by setting 
\begin{equation*}\label{Eq: lattice triple}
	\Delta(\mathcal{L}) := [\mathcal{L},\kappa_{L}, H_{L}]   \in   K_{0}(\mathbb{Z}[G], \mathbb{Q}^{c}[G]).
\end{equation*}
In particular, if $\mathcal{A}_{L/K}$ exists, then (by Proposition~\ref{Lemma: weakly rami proj iff}) it gives rise to a well-defined element $\Delta(\mathcal{A}_{L/K})$. 

In the sequel, we shall investigate relations between this element and the element of $\zeta(\QQ^c[G])^\times$ obtained by setting 
\[ T^{(2)}_{L/K} := (\tau^{G}_{K}\cdot (\psi_{2, *}  -1 )(\tau'_{L/K}))^{-1},\]
where elements $\tau^{G}_{K}$ and $\tau'_{L/K}$ are as defined in \S\ref{S: modified GGS}.

As a first step, we prove the following result.  

\begin{lemma}\label{Prop: global a in Q} Suppose that $\mathcal{A}_{L/K}$ exists and fix an element $x$ of $K_1(\QQ^c[G])$ such that the difference $\partial^1_{\ZZ,\QQ^c,G}(x) - 
	\Delta(\mathcal{A}_{L/K})$ belongs to the subgroup $K_0(\ZZ[G],\QQ[G])$ of $K_0(\ZZ[G],\QQ^c[G])$. Then, one has 
	\[ {\rm Nrd}_{\QQ^c[G]}(x)\cdot T^{(2)}_{L/K} \in \zeta(\QQ[G])^\times.\]
\end{lemma}
\begin{proof}  We recall (from Proposition~\ref{Prop: JG sums equivariant}) that $J_{2, L/K}$ belongs to $\zeta(\mathbb{Q}[G])^\times$. It is also clear that (by the explicit definition of $y_{L/K}$) the element 
	$(1- \psi_{2, *})(y_{L/K})$ belongs to $\zeta(\mathbb{Q}[G])^\times$. Upon recalling the equality of Proposition \ref{Prop: JG decomp}(iii), we are therefore reduced to showing that, for the specified element $x$, the element 
	${\rm Nrd}_{\QQ^c[G]}(x)\cdot (\tau_{L/K}^{\dagger})^{-1}$ belongs to $\zeta(\QQ[G])^\times$.
	
	To do this, we fix an isomorphism $\lambda: \QQ\otimes_\ZZ H_{L} \cong L$ of $\QQ[G]$-modules. Then, the automorphism $\kappa_L\circ (\QQ^c\otimes_\QQ\lambda)$ of $\QQ^c\otimes_\ZZ H_L$ defines an element $\langle \lambda\rangle$ of $K_1(\QQ^c[G])$, for which one has 
	\[\partial_{\ZZ,\QQ^c,G}^1(\langle \lambda\rangle) = \bigl[H_L,\kappa_L\circ (\QQ^c\otimes_\QQ \lambda ),H_L\bigr]
	=\bigl[ H_L,(\QQ^c\otimes_\QQ \lambda ),\mathcal{A}_{L/K}\bigr] +  \Delta(\mathcal{A}_{L/K}). \]
%\begin{align*}\partial_{\ZZ,\QQ^c,G}^1(\langle \lambda\rangle) =&\, \bigl[H_L,\kappa_L\circ (\QQ^c\otimes_\QQ \lambda ),H_L\bigr]\\
%=&\, \bigl[ H_L,(\QQ^c\otimes_\QQ \lambda ),\mathcal{A}_{L/K}\bigr] +  \Delta(\mathcal{A}_{L/K}).\end{align*}
%
	In particular, since the element $ \bigl[ H_L,(\QQ^c\otimes_\QQ \lambda),\mathcal{A}_{L/K}\bigr] = \bigl[ H_L,\lambda,\mathcal{A}_{L/K}\bigr]\in K_0(\ZZ[G],\QQ[G])$, the above two equalities can therefore be combined with the assumption that $\partial^1_{\ZZ,\QQ^c,G}(x) - 
	\Delta(\mathcal{A}_{L/K}) \in K_0(\ZZ[G],\QQ[G])$ to imply that $\partial_{\ZZ,\QQ^c,G}^1(\langle \lambda\rangle\cdot x^{-1})$ belongs to $K_0(\ZZ[G],\QQ[G])$. 
	We also note that, the commutativity of the diagram (\ref{Diagram: Kseq}) implies that if there is an element $\alpha \in K_1(\QQ^c[G])$ such that the image $\partial_{\ZZ,\QQ^c,G}^1(\alpha)$ belongs to $K_0(\ZZ[G],\QQ[G])$, then $\alpha$ must belong to the image of $K_1(\QQ[G])$ in $K_1(\QQ^c[G])$. Given this argument, we can therefore deduce that $\langle \lambda\rangle$ differs from $x$ by an element of $K_1(\QQ[G])$ and so it is enough for us to show that
	\[ {\rm Nrd}_{\QQ^c[G]}(\langle \lambda\rangle)\cdot (\tau_{L/K}^{\dagger})^{-1} \in \zeta(\QQ[G])^\times. \]
	This is proved by Burns and Bley in \cite[Prop. 3.4 and Rem. 3.5]{BB}. 
	%(in which $\kappa_{L}$ and $\lambda$ corresponeds to $\rho_L$ and $\psi_{\CC}$ respectively ). 
	%
	%\[ [\mathcal{A}_{L/K},\kappa_{L}, H_{L}] - \delta_{G}(\tau_{L/K}^{\dagger}) \in K_{0}(\mathbb{Z}[G] , \mathbb{Q}[G]) . 
	%]
	%
\end{proof}

The result of Lemma \ref{Prop: global a in Q} allows us to make the following definition. In this definition, we use the extended boundary homomorphism $\delta_{G}: \zeta(\QQ[G])^\times \to K_0(\ZZ[G],\QQ[G])$ from Lemma \ref{ext bound hom}.

\begin{definition}\label{Def: frak a global} Suppose $\mathcal{A}_{L/K}$ exists and fix an element $x$ of $K_1(\QQ^c[G])$ such that the difference $\partial^1_{\ZZ,\QQ^c,G}(x) - 
	\Delta(\mathcal{A}_{L/K})$ belongs to $K_0(\ZZ[G],\QQ[G])$. Then, the `canonical relative element' of Bley, Burns and Hahn is the element of $K_{0}(\mathbb{Z}[G], \mathbb{Q}[G])$ that is obtained by setting
	\[\mathfrak{a}_{L/K}:= \Delta(\mathcal{A}_{L/K}) - \partial^1_{\ZZ,\QQ^c,G}(x) + \delta_{G}
	\bigl({\rm Nrd}_{\QQ^c[G]}(x)\cdot T^{(2)}_{L/K}\bigr).\]
\end{definition}

The basic properties of this element are recorded in the following result. We note that claim (ii) of this result shows that the element $\mathfrak{a}_{L/K}$ defined above does generalise that defined (for odd degree extensions) by Bley, Burns and Hahn in \cite{BBH}. 

In claim (iii) of the next result, we use the following notation: for each prime $\ell$ and embedding of fields $j^c_{\ell}:\mathbb{Q}^{c}\rightarrow \mathbb{Q}^{c}_{\ell}$, we also write $j^c_\ell$ for the induced homomorphism of rings $\QQ^c[G] \to \QQ^c_\ell[G]$, and we consider the homomorphism of abelian groups 
\begin{align}\label{jlc def}
j^c_{\ell, *}: K_{0}(\mathbb{Z}[\Gamma], \mathbb{Q}^c[\Gamma]) \, & \to K_{0}(\mathbb{Z}_\ell[\Gamma],\mathbb{Q}_\ell^c[\Gamma]), \nonumber \\
[P, \phi, Q] \, & \mapsto  [P_\ell, \mathbb{Q}_\ell^c\otimes_{\mathbb{Q}^c, j_\ell}\phi, Q_\ell],
\end{align}
where $\mathbb{Q}_\ell^c\otimes_{\mathbb{Q}^c,  j_\ell}\phi$ denotes the composite isomorphism 
\[ \mathbb{Q}_\ell^c\otimes_{\mathbb{Z}_\ell}(\mathbb{Z}_\ell\otimes_{\mathbb{Z}} P) \xrightarrow{\cong}  \mathbb{Q}_\ell^c\otimes_{\mathbb{Q}^c} (\mathbb{Q}^c\otimes_{\mathbb{Z}} P) \xrightarrow{\mathbb{Q}_\ell^c\otimes_{\mathbb{Q}^c , j_\ell} \phi } \mathbb{Q}_\ell^c\otimes_{\mathbb{Q}^c} (\mathbb{Q}^c\otimes_{\mathbb{Z}} Q) \xrightarrow{\cong} \mathbb{Q}_\ell^c\otimes_{\mathbb{Z}_\ell}(\mathbb{Z}_\ell\otimes_{\mathbb{Z}} Q),\]
with $\mathbb{Q}_\ell^c\otimes_{\mathbb{Q}^c,  j_\ell}-$ the tensor product obtained by using $j^c_\ell$ to regard $\mathbb{Q}_\ell^c$ as a $\mathbb{Q}^c$-module.

\begin{prop}\label{frak a global independence} Suppose $\mathcal{A}_{L/K}$ exists. Then, the following claims are valid. 
	\begin{itemize}
		\item[(i)] $\mathfrak{a}_{L/K}$ is independent of the choice of $x$. 
		\item[(ii)] If $|G|$ is odd, then $\mathfrak{a}_{L/K}$ is equal to the element defined by Bley, Burns and Hahn in \cite[\S5]{BBH}.
		\item[(iii)] Fix a prime number $\ell$ and an embedding $j^c_{\ell}:\mathbb{Q}^{c}\rightarrow \mathbb{Q}^{c}_{\ell}$. Then, in $K_0(\ZZ_\ell[G],\QQ_\ell[G])$, one has 
		\[ j_{\ell,\ast}(\mathfrak{a}_{L/K}) = j^c_{\ell,\ast}(\Delta(\mathcal{A}_{L/K})) + \delta_{\ZZ_\ell,\QQ_\ell^c,G}( j^c_\ell(T^{(2)}_{L/K})),\]
		where $j_{\ell,*}$ denotes the natural projection homomorphism $K_0(\ZZ[\Gamma],\QQ[\Gamma]) \to K_0(\ZZ_\ell[\Gamma],\QQ_\ell[\Gamma])$ in \eqref{K_0 proj K_0_p} and $\delta_{\ZZ_\ell,\QQ_\ell^c,G}$ denotes the composite homomorphism 
		\[ \partial^1_{\ZZ_\ell,\QQ^c_\ell,G}\circ ({\rm Nrd}_{\QQ^c_\ell[G]})^{-1}: \zeta(\QQ_\ell^c[G])^\times\to K_0(\ZZ_\ell[G],\QQ_\ell^c[G]).\]
		\item[(iv)] For any element $x$ as in the definition of $\mathfrak{a}_{L/K}$, we define an element of ${\rm Cl}(\ZZ[G])$ by setting  
		\[ W^{(2)}_{L/K} :=  \partial^0_{\ZZ,\QQ,G}\bigl(\delta_G({\rm Nrd}_{\QQ^c[G]}(x)\cdot T^{(2)}_{L/K})\bigr).\]
		This element is independent of $x$, of order dividing $2$ and vanishes if $G$ has no irreducible symplectic characters. In addition, one has  
		\[  \partial^0_{\ZZ,\QQ,G}(\mathfrak{a}_{L/K}) = [\mathcal{A}_{L/K}] + W^{(2)}_{L/K}.\]
	\end{itemize}
\end{prop}

\begin{proof} Throughout this argument, we fix an element $x$ of $K_1(\QQ^c[G])$ as specified in the definition of $\mathfrak{a}_{L/K}$. We also set $y := T^{(2)}_{L/K}$. 
	
	To prove claim (i), we choose another element $x'$ of $K_1(\QQ^c[G])$ that has the same property as $x$. Then, one has 
	\[ \partial_{\ZZ,\QQ^c,G}^1(x^{-1}x') = \partial_{\ZZ,\QQ^c,G}^1(x) - \partial_{\ZZ,\QQ^c,G}^1(x')\in K_0(\ZZ[G],\QQ[G]),\]
	and so the diagram (\ref{Diagram: Kseq}) implies that $x^{-1}x'$ belongs to $K_1(\QQ[G])$. It follows that  
	\begin{align*}
		\delta_G({\rm Nrd}_{\QQ^c[G]}(x')y) 
		=\,&\delta_G ({\rm Nrd}_{\QQ^c[G]}(x'x^{-1}){\rm Nrd}_{\QQ^c[G]}(x)y)\\
		=\, &\delta_G ({\rm Nrd}_{\QQ[G]}(x'x^{-1}){\rm Nrd}_{\QQ^c[G]}(x)y)\\
		=\, &\delta_G ({\rm Nrd}_{\QQ[G]}(x'x^{-1})) + \delta_G({\rm Nrd}_{\QQ^c[G]}(x)y)\\
		=\, &  \partial^{1}_{\ZZ,\QQ,G}(x'x^{-1}) +  \delta_G({\rm Nrd}_{\QQ^c[G]}(x)y)\\
		=\, &  \partial^{1}_{\ZZ,\QQ^c,G}(x') - \partial^{1}_{\ZZ,\QQ^c,G}(x) +  \delta_G({\rm Nrd}_{\QQ^c[G]}(x)y),\end{align*}
	where the fourth equality is given by Lemma \ref{ext bound hom}(i), and the last one again follows from the commutativity of the diagram (\ref{Diagram: Kseq}). This proves claim (i).
	%and so claim (i) is true since Lemma \ref{ext bound hom}(i) implies that 
	%
	%\[ \delta_G ({\rm Nrd}_{\QQ^c[G]}(x'x^{-1})) = \partial^{1}_{\ZZ,\QQ,G}(x'x^{-1}) = \partial^{1}_{\ZZ,\QQ^c,G}(x') - \partial^{1}_{\ZZ,\QQ^c,G}(x).\]
	
	For claim (ii), we first note that if $G$ has odd order, then $\zeta(\QQ[G])^{\times + } = \zeta(\QQ[G])^\times$ (see Lemma~\ref{Prop: Nrd bijective cond}(ii)). Therefore, in this case, Lemma \ref{ext bound hom}(ii) implies that 
	\begin{align*}
		\delta_G ({\rm Nrd}_{\QQ^c[G]}(x)y) 
		=\, & \partial^1_{\ZZ,\QQ,G} \bigl(({\rm Nrd}_{\QQ[G]})^{-1} ( {\rm Nrd}_{\QQ^c[G]}(x)y ) \bigr)\\
		=\, & \partial^1_{\ZZ,\QQ^c,G}\bigl(x\cdot({\rm Nrd}_{\QQ^c[G]})^{-1}(y)\bigr)\\ 
		=\, &  \partial^1_{\ZZ,\QQ^c,G}(x) + \partial^1_{\ZZ,\QQ^c,G}\bigl(({\rm Nrd}_{\QQ^c[G]})^{-1}(y)\bigr)
	\end{align*}
	and hence that $\mathfrak{a}_{L/K}$ is equal to the element $\Delta(\mathcal{A}_{L/K}) + \partial^1_{\ZZ,\QQ^c,G}\bigl(({\rm Nrd}_{\QQ^c[G]})^{-1}(y)\bigr)$, which is considered by Bley, Burns and Hahn in loc. cit.  
	
	Turning to claim (iii), we write $j^{c,1}_{\ell,\ast}$ for the homomorphism $K_1(\QQ^c[G]) \to K_1(\QQ_\ell^c[G])$ that is induced by the embedding $j_\ell^c$. Then, Lemma \ref{ext bound hom}(iii) implies that 
	\begin{align*} 
		j_{\ell, *}\bigl(\delta_{G}({\rm Nrd}_{\QQ^c[G]}(x)y)\bigr) 
		=&\, \partial_{\ZZ_\ell,\QQ_\ell,G} \left( ( {\rm Nrd}_{\QQ_\ell[G]})^{-1}  \bigl(j_\ell^c({\rm Nrd}_{\QQ^c[G]}(x)y) \bigr)\right)\\
		%\bigl(({\rm Nrd}_{\QQ_\ell[G]})^{-1}(j_\ell^c({\rm Nrd}_{\QQ^c[G]}(x)y))\bigr)\bigr)\\
		=&\, \partial_{\ZZ_\ell,\QQ^c_\ell,G}\bigl(j^{c,1}_{\ell,\ast}(x)\cdot ({\rm Nrd}_{\QQ^c_\ell[G]})^{-1}(j_\ell^c(y))\bigr)\\
		=&\, \partial_{\ZZ_\ell,\QQ^c_\ell,G}\bigl(j^{c,1}_{\ell,\ast}(x)\bigr) + \delta_{\ZZ_\ell,\QQ_\ell^c,G}\bigl(j_\ell^c(y)\bigr)\\
		=&\, j_{\ell,\ast}^c\bigl(\partial_{\ZZ,\QQ^c,G}(x)\bigr) + \delta_{\ZZ_\ell,\QQ_\ell^c,G}\bigl(j_\ell^c(y)\bigr).\end{align*} 
	For the last equality, one can check directly from the explicit definition of \eqref{jlc def} and the connecting homomorphism in \eqref{Diagram: Kseq} that $\partial_{\ZZ_\ell,\QQ^c_\ell,G} \circ j^{c,1}_{\ell,\ast} = j_{\ell,\ast}^c \circ \partial_{\ZZ,\QQ^c,G}$. Thus, the equality in claim (iii) follows directly from the definition of $\mathfrak{a}_{L/K}$ as an explicit sum.  
	
	Finally, we note that the second displayed equality in claim (iv) is true since the explicit definition of $\mathfrak{a}_{L/K}$ implies that  
	\begin{align*}  \partial^0_{\ZZ,\QQ,G}(\mathfrak{a}_{L/K}) 
		=&\,   \partial^0_{\ZZ,\QQ^c,\Gamma}(\Delta(\mathcal{A}_{L/K})) -  \partial^0_{\ZZ,\QQ^c,G}(\partial^1_{\ZZ,\QQ^c,G}(x)) +  \partial^0_{\ZZ,\QQ,G}\bigl(\delta_{G}
		\bigl({\rm Nrd}_{\QQ^c[G]}(x)y\bigr)\bigr)\\
		=&\,  \partial^0_{\ZZ,\QQ^c,G}\bigl([\mathcal{A}_{L/K}, \kappa_L,H_L]\bigr) + W^{(2)}_{L/K}\\
		=&\, [\mathcal{A}_{L/K}] - [H_L] + W^{(2)}_{L/K} \\
		=&\, [\mathcal{A}_{L/K}] + W^{(2)}_{L/K}.\end{align*}
	Here, the second equality follows from the explicit definition of $W^{(2)}_{L/K}$ and the fact that the exactness of (\ref{Diagram: Kseq}) implies that the composite map $\partial^0_{\ZZ,\QQ^c,G}\circ \partial^1_{\ZZ,\QQ^c,G}$ is zero. The third equality follows from the explicit definition of the connecting homomorphism $\partial^0_{\ZZ,\QQ^c,G}$, and the last equality is given by the fact that $H_L$ is a free $\ZZ[G]$-module.  
	
	Since both elements $\mathfrak{a}_{L/K}$ and $\mathcal{A}_{L/K}$ are independent of $x$ (the first by claim (i) and the second is obvious), the above equality implies that $W^{(2)}_{L/K}$ is independent of $x$. 
	
	To prove that $2\cdot W^{(2)}_{L/K} = 0$, we first note that, by Lemma \ref{Prop: Nrd bijective cond}(i), $({\rm Nrd}_{\QQ^c[G]}(x)y)^2$ belongs to $\zeta(\QQ[G])^{\times +}$ and hence, one has
\begin{align*}  
2\cdot W^{(2)}_{L/K} =&\,  \partial^0_{\ZZ,\QQ,G}\bigl(\delta_G(({\rm Nrd}_{\QQ^c[G]}(x)y)^2)\bigr)\\
=&\, \partial^0_{\ZZ,\QQ,G} \left(  \partial_{\ZZ,\QQ,G}^1 \bigl( ({\rm Nrd}_{\QQ[G]})^{-1} ( ({\rm Nrd}_{\QQ^c[G]}(x)y)^2)\bigr)\right)\\
=&\, 0. \end{align*}
	Here, the second quality follows from Lemma \ref{ext bound hom}(ii), and the last one follows from the fact that $\partial^0_{\ZZ,\QQ,G}\circ \partial^1_{\ZZ,\QQ,G}$ is equal to the zero map.  
	
	In a similar way, one can deduce that $W^{(2)}_{L/K}$ vanishes whenever $G$ has no irreducible symplectic characters from the results of Lemma \ref{ext bound hom}(ii) and Lemma \ref{Prop: Nrd bijective cond}(ii)(b).
\end{proof} 

\begin{remark} A natural problem is to explicitly describe the difference between the class $W^{(2)}_{L/K}$ defined in 
	Proposition \ref{frak a global independence}(iv) and the Cassou-Nogu\`es-Fr\"ohlich root number class $W_{L/K}$, which plays a key role in classical Galois module theory. We shall consider the problem again in \S\ref{S: weakly classgroup}. \end{remark}

\subsubsection{The local element}\label{def local element}
Fix a finite Galois extension $E/F$ of local fields of residue characteristic $\ell$ and set $\Gamma := {\rm Gal}(E/F)$, we use the exactly same recipe \cite[\S7A]{BBH} of Bley, Burn and Hahn to define the canonical local relative element of $K_{0}(\mathbb{Z}_{\ell}[\Gamma], \mathbb{Q}_{\ell}^{c}[\Gamma])$.

\begin{definition}\label{local a def} 
	We define an element of $K_{0}(\mathbb{Z}_{\ell}[\Gamma], \mathbb{Q}_{\ell}^{c}[\Gamma])$ by setting 
	\[\mathfrak{a}_{E/F}:= \Delta(\mathcal{A}_{E/F}) - \delta_{\ZZ_\ell,\QQ_\ell^c,\Gamma}(j^c_{\ell}(T^{(2)}_{E/F})) - U_{E/F},  \]
where $\delta_{\ZZ_\ell,\QQ_\ell^c,\Gamma}:= \partial_{\ZZ_\ell,\QQ^c_\ell,\Gamma} \circ ({\rm Nrd}_{\QQ^c_\ell[\Gamma]} )^{-1}$, $j^c_\ell$ is the ring embedding $\zeta(\QQ^c[\Gamma]) \rightarrow \zeta(\QQ^c_\ell[\Gamma])$ induced by a choice of field embedding $\QQ^c\to \QQ_\ell^c$ and $U_{E/F}$ is the canonical `unramified' element of $K_{0}(\mathbb{Z}_{\ell}[\Gamma], \mathbb{Q}_{\ell}^{c}[\Gamma])$ that is defined by Breuning in \cite{B04}.
\end{definition}

Here the element $\Delta(\mathcal{A}_{E/F})$ an $T^{(2)}_{E/F}$ are defined to be the local analogues of the elements defined in the last section.

Next, for each element $a$ in $E$ generating a normal basis of $E/F$, and each character $\chi$ of representation $T_\chi : \Gamma \rightarrow {\rm GL}_{n}(\QQ_\ell^c)$ (here we consider $\QQ_\ell^c$ for the coefficient field for group representations), we recall from \cite[Chap. I, \S4 and Chap. III, \S3, (3.1)]{F83} the definitions of the \textit{resolvent} element and the (local) \emph{norm resolvent}
\begin{equation*}\label{Eq: resolvent def}
	(a|\chi):=\mathrm{Det}(\sum_{g\in \Gamma} g(a) T_{\chi}(g^{-1})),  \quad \mathcal{N}_{F/\mathbb{Q}_\ell}(a|\chi)= \prod_{\omega}(a|\chi^{\omega^{-1}})^{\omega}, 
\end{equation*}
where in the second term $\omega$ products through a transversal of $\Omega_{F}$ in $\Omega_{\mathbb{Q}_\ell}$. 

We now follow Breuning in giving an explicit description of the element $\mathfrak{a}_{E/F}$ in terms of the norm resolvents. To do this, we fix a $\mathbb{Z}_\ell$-basis $\{a_\sigma \}_{\Sigma(F)}$ of $\mathcal{O}_F$ and set $\delta_{F}:=\det(\tau(a_\sigma) )_{ \tau, \sigma \in \Sigma(F)} \in \QQ_\ell^c$.

\begin{prop}\label{Prop: local in Q}\ %[{\cite[Prop. 7.1]{BBH}}]
	\begin{itemize}
		\item[(i)] $\mathfrak{a}_{E/F}$ is independent of the choice of embedding $j_\ell^c$ and belongs to $K_{0}(\mathbb{Z}_{\ell}[\Gamma], \mathbb{Q}_{\ell}[\Gamma])$.
		\item[(ii)] (Breuning) Fix an element $a$ of $E$ such that 
		$\mathcal{A}_{E/F}= \mathcal{O}_{F}[\Gamma]\cdot a$. Then, one has 
		\[ \sum_{\chi\in \widehat{\Gamma}} (\delta_{F}^{\chi(1)} \cdot \mathcal{N}_{F/\mathbb{Q}_{\ell}}(a|\chi))e_{\chi}\in \zeta(\QQ^c_\ell[\Gamma])^\times\]
		and, in $K_{0}(\mathbb{Z}_{\ell}[\Gamma], \mathbb{Q}^c_{\ell}[\Gamma])$, one has 
		\[ \Delta(\mathcal{A}_{E/F}) = \delta_{\ZZ_\ell,\QQ_\ell^c,\Gamma} \left(\sum_{\chi\in \widehat{\Gamma}} (\delta_{F}^{\chi(1)} \cdot \mathcal{N}_{F/\mathbb{Q}_{\ell}}(a|\chi))e_{\chi} \right). \]
		\item[(iii)] (Breuning) For each embedding $k:\QQ^c \to \QQ_\ell^c$, element $\delta_F/k(\tau_F)$ belongs to $(\mathcal{O}_\ell^t)^\times$. 
		\item[(iv)] (Breuning) Let $j_{\ell, *}^{t}$ be the homomorphism defined in \eqref{Eq: Taylor map def0}, then $U_{E/F} \in \ker(j_{\ell, *}^{t})$.
	\end{itemize}
\end{prop}

\begin{proof} Since the terms $\Delta(\mathcal{A}_{E/F})$ and $U_{E/F}$ obviously do not depend on $j_\ell^c$, the element $\mathfrak{a}_{E/F}$ is independent of $j_\ell^c$ if $j_\ell^c(T^{(2)}_{E/F})$ is independent of the choice. 
	
In addition, from the equality in Remark \ref{Rem: loc JG sums def}(ii), one has $	T^{(2)}_{E/F} = \tau_{E/F}^{\dagger} \cdot (J_{2, E/F}\cdot (1 - \psi_{2, *})(y_{E/F})) $. We also note that Remark~\ref{Rem: loc JG sums def}(i) and the explicit definition of $y_{E/F}$ combine to imply that 
	\begin{equation}\label{convenient containment} (J_{2, E/F}\cdot(1 - \psi_{2, *})(y_{E/F})) )\in \zeta(\mathbb{Q}[\Gamma])^{\times}.\end{equation}
	It is therefore sufficient to show that $\tau_{E/F}^{\dagger}$ is independent of the choice of $j^c_{\ell}$, and this is proved by Breuning in \cite[Lem. 2.2]{B04}.
	
	Thus, to prove claim (i), it is enough to show that $\mathfrak{a}_{E/F}$ belongs to $K_{0}(\mathbb{Z}_{\ell}[\Gamma], \mathbb{Q}_{\ell}[\Gamma])$, and this follows directly from (\ref{convenient containment}) and the containment 
	\[\Delta(\mathcal{A}_{E/F})- \delta_{\ZZ_\ell,\QQ_\ell^c,\Gamma}(j^c_{\ell}(\tau_{E/F}^{\dagger})) - U_{E/F} \in K_{0}(\mathbb{Z}_{\ell}[\Gamma], \mathbb{Q}_{\ell}[\Gamma]) \]
	that is proved by Bley, Burns and Hahn in \cite[Prop. 7.1]{BBH}. 
	
	Claims (ii), (iii) and (iv) are proved by Breuning in \cite[Lem. 4.16, Lem. 4.29 and Lem. 4.4]{B_phd} respectively.  
	%Lemma 4.16, Lemma 4.29 and Proposition 4.4 of \cite{B_phd} 
\end{proof}

\begin{remark}\label{frak a' def}
	Let $\QQ_\ell^t$ denote the maximal tamely ramified extension of $\QQ_\ell$ and $\mathcal{O}_\ell^t$ denote the valuation ring of $\QQ_\ell^t$. In the sequel, (as a consequence of the Taylor's Fixed Point Theorem, which is discussed in \S\ref{S: taylor fixed point}) we will often consider the image in $K_0(\mathcal{O}_\ell^t[\Gamma], \QQ_\ell^c[\Gamma])$ of the element $\mathfrak{a}_{E/F}$. To help address the problem, we set
	\[\mathfrak{a}'_{E/F}:= \Delta(\mathcal{A}_{E/F}) - \delta_{\ZZ_\ell,\QQ_\ell^c,\Gamma}(j^c_{\ell}(T^{(2)}_{E/F})).  \]
	Then, Definition~\ref{local a def} and Proposition~\ref{Prop: local in Q}(iv) combine to imply that the images in $K_0(\mathcal{O}_\ell^t[\Gamma], \QQ_\ell^c[\Gamma])$ of the elements $\mathfrak{a}_{E/F}$ and $\mathfrak{a}'_{E/F}$ coincide. We will use this fact frequently in the sequel.
\end{remark}

\subsubsection{From global to local}

By adapting an argument of Bley, Burns and Hahn, we shall now describe the precise connection between  the elements that are defined in \S~\ref{def global element} for number fields and in \S\ref{def local element} for local fields. 

To do this, we fix a weakly ramified Galois extension $L/K$ of number fields for which $\mathcal{A}_{L/K}$ exists, and we set $G:=\mathrm{Gal}(L/K)$. For each place $v$ of $K$, we fix a place $w$ of $L$ lying above $v$ and identify the Galois group of $L_w/K_v$ with the decomposition subgroup $G_w$ of $w$ in $G$.

We recall the induction homomorphisms on relative $K$-groups that are defined in \eqref{Eq: K_0 ind def}. For each prime $\ell$, we identify $K_0(\ZZ_\ell[G],\QQ_\ell[G])$ with a subgroup of $K_0(\ZZ[G],\QQ[G])$ by means of the decomposition (\ref{Eq: KT iso K0}). 

\begin{prop}\label{Lemma: frak a local decomp}
	Let $L/K$ be a weakly ramified finite Galois extension of number fields, with $G:=\mathrm{Gal}(L/K)$, for which $\mathcal{A}_{L/K}$ exists. Then, in $K_0(\ZZ[G],\QQ[G])$, one has 
	\[\mathfrak{a}_{L/K} = \sum_{\ell}\sum_{v|\ell} {\rm i}_{G_{w}, \mathbb{Q}_{\ell}}^{G, *} (\mathfrak{a}_{L_{w}/K_{v}}), \]
	where the first sum runs over all rational primes $\ell$ and the second over all places $v$ of $K$ of residue characteristic $\ell$. %each place $v\in K$ we fix $w$ in $L$ above $v$ and identify the decomposition subgroup $G_w$ with $\mathrm{Gal}(L_{w}/K_v)$.
\end{prop}

\begin{proof}
The proof is exactly the same to \cite[Th. 7.6]{BBH} using Proposition~\ref{frak a global independence}(iii) and for the last line of loc.cit that `the image under $j_{\ell, *}^t \circ \delta_{G, \ell}$ of the latter element vanishes as a consequence of', one replaces the referred results by our Lemma~\ref{Prop: decomp pre prop} following.
\end{proof}

\begin{lemma}\label{Prop: decomp pre prop}
If $v$ is a non-Archimedean place of $K$ that does not divide $\ell$, then the element 
		\[ j_{\ell}^c ((\psi_{2, *} -2)(\tau_{L_{w}/K_{v}} \cdot y^{-1}_{L_{w}/K_{v}})) \in {\rm Det}(\mathcal{O}_\ell^t[G_w]^{\times}).\]
\end{lemma}

\begin{proof} We shall adapt an argument of Agboola, Burns, Caputo and the present author in \cite[Prop. 6.6(a)]{AC}. To do this, we write $p$ for the residue characteristic of $v$ (so that $p \not= \ell$) and $\mathbb{Q}(p^\infty)$ for the subextension of $\mathbb{Q}^c$ that is generated by all $p$-power order roots of unity. We also fix an $\ell$-adic place of $\mathbb{Q}^c$, and for each finite extension $M$ of $\QQ$ in $\QQ(p^\infty)$, we write $M_\ell$ for its completion at the fixed $\ell$-adic place and $\mathcal{O}_{M_\ell}$ for the valuation ring of $M_\ell$. 
	
	Now, since $p\neq \ell$, the result \cite[(3.3b$'$)]{HW} of Holland and Wilson implies that, for a large enough subfield $M$ of $\QQ(p^\infty)$, one has $j^c_{\ell}(\tau_{L_{w}/K_{v}} y^{-1}_{L_{w}/K_{v}})\in \mathrm{Det}(\mathcal{O}_{M_\ell}[G_{w}]^{\times})$. It follows that
	\begin{multline*}j^c_{\ell}((\psi_{2, *}-2)(\tau_{L_{w}/K_{v}} y^{-1}_{L_{w}/K_{v}})) = (\psi_{2, *}-2)(j^c_\ell(\tau_{L_{w}/K_{v}} y^{-1}_{L_{w}/K_{v}}))\\ 
		\in \psi_{2}(\mathrm{Det}(\mathcal{O}_{M_\ell}[G_{w}]^{\times}) ) + \mathrm{Det}(\mathcal{O}_{M_\ell}[G_{w}]^{\times}) =  \mathrm{Det}(\mathcal{O}_{M_\ell}[G_{w}]^{\times}), \end{multline*}
	where the equality is a consequence of the fact that, since $M_\ell$ is an unramified extension of $\mathbb{Q}_\ell$, the result \cite[Th. 1$'$]{Nog_Tay_85} of Cassou-Nogues and Taylor proves that $\psi_{2}(\mathrm{Det}(\mathcal{O}_{M_\ell}[G_{w}]^{\times})  \subseteq \mathrm{Det}(\mathcal{O}_{M_\ell}[G_{w}]^{\times})$. Therefore, the result follows from the obvious inclusion
	$M_\ell \subseteq \QQ_\ell^t$ and, hence, $\mathrm{Det}(\mathcal{O}_{M_\ell}[G_{w}]^{\times}) \subseteq \mathrm{Det}(\mathcal{O}_{\ell}^t[G_{w}]^{\times})$.
\end{proof}

\subsection{The conjecture of Bley, Burns and Hahn}\label{S: BBH conj}
In this section, we recall the central conjecture that is formulated by Bley, Burns and Hahn in \cite{BBH}. To do this, 
we first recall a variant of the classical `unramified characteristic' that is introduced in loc. cit.

\subsubsection{The twisted idelic unramified characteristic}\label{tiuc section}
To start, we consider the case of local fields. 

\begin{definition}\label{Def: loc c} Let $E/F$ be a finite Galois extension of $\ell$-adic fields and set $\Gamma := \Gal(E/F)$. Then, the `twisted unramified characteristic' of $E/F$ is the element of $K_{0}(\mathbb{Z}_{\ell}[\Gamma], \mathbb{Q}_{\ell}[\Gamma])$ that is obtained by setting
	\[\mathfrak{c}_{E/F}:= \delta_{\ZZ_\ell,\QQ_\ell,\Gamma}((1 - \psi_{2, *})(y_{E/F}))  , \]
	where $y_{E/F}$ is the equivariant unramified characteristic, and  $\delta_{\ZZ_\ell,\QQ_\ell,\Gamma}$ denotes the composite homomorphism 
	\begin{equation}\label{frak c local delta}
		\partial^1_{\ZZ_\ell,\QQ_\ell,\Gamma}\circ ({\rm Nrd}_{\QQ_\ell[\Gamma]})^{-1}: \zeta(\QQ_\ell[\Gamma])^\times\to K_0(\ZZ_\ell[G],\QQ_\ell[\Gamma]).
	\end{equation}
	
\end{definition}

The following result extends the observations that are made in \cite[Rem. 7.5]{BBH}. 

\begin{lemma}\label{Prop: local c tame vanishes} Let $E/F$ and $\Gamma$ be as above, suppose that the inertia group $I := \Gamma_0$ of $\Gamma$ has odd order, the following claims are valid.
	
	\begin{itemize}
		\item[(i)] If $\Gamma$ has odd order or $\Gamma$ is abelian if its order is even. Then, in $\zeta(\QQ[\Gamma])^\times$, one has  
		\[(1 - \psi_{2, *})(y_{E/F}) = (1-e_I) + \sigma^{-1}e_I, \]
		where $\sigma$ is a lift to $\Gamma$ of the Frobenius element in $\Gamma/I$ and $e_I$ is the idempotent $|I|^{-1}\sum_{g\in I} g$.
		\item[(ii)] The element $\mathfrak{c}_{E/F}$ vanishes if either $E/F$ is tamely ramified or both totally ramified and of odd degree. 
	\end{itemize}
\end{lemma}

\begin{proof} Assume $\Gamma$ is either abelian or of odd order. Then, in this case, if $\phi$ is an irreducible character of $\Gamma$, so is $\psi_2(\phi)$ (see Proposition~\ref{Prop: adam}(iv) for the case that $\Gamma$ has odd order, and for the case that $\Gamma$ is abelian, one has $\phi(1) = \psi_2(\phi)(1) = 1$). 
	
	In particular, if $\phi$ is both irreducible and unramified (that is, trivial on $I$), then $\phi$ is linear, and so $\psi_{2}(\phi)$ is also both linear and unramified. Thus, one has  	
	\begin{align}\label{first eq} (1 - \psi_{2, *})y(F, \phi) =&\,  y(F, \phi- \psi_{2}(\phi)) \notag \\
		=&\, y(F, \phi)y(F,\psi_{2}(\phi))^{-1}\notag\\
		=&\, ((-1)^{\phi(1)}{\rm det}_\phi(\sigma))((-1)^{\psi_2(\phi)(1)}{\rm det}_{\psi_2(\phi)}(\sigma))^{-1}\notag\\
		=&\, \phi(\sigma)(\phi(\sigma^2))^{-1}\notag\\
		=&\, \phi(\sigma)^{-1}.
	\end{align}
	The fourth equality here follows from Proposition~\ref{Prop: adam}(iii) and the fact that $\phi(1) = \psi_2(\phi)(1) = 1$. 
	
	Now we also assume that $I$ has odd order, then one has $I = \{g^2: g \in I\}$. In this case, if an irreducible character $\phi$ is ramified (that is, non-trivial on $I$), then $\psi_2(\phi)$ is also ramified, and so the unramified parts (see Definition~\ref{Def: unrami char}(ii)) $n(\phi)$ and $n(\psi_2(\phi))$ both vanish. Hence, one has    
	\begin{equation}\label{second eq} 
		(1 - \psi_{2, *})y(F,\phi) = y(F, \phi- \psi_{2}(\phi)) = 
		y(F, \phi)y(F,\psi_{2}(\phi))^{-1} = 1\cdot 1 = 1. 
	\end{equation}

	Write $\widehat{\Gamma}_R$ and $\widehat{\Gamma}_U$ for the sets of ramified and unramified irreducible characters of $\Gamma$. Then, by putting together the formulas (\ref{first eq}) and (\ref{second eq}), we can compute that 
	\begin{align*}\label{Eq: twisted unrami char y} 
		(1-\psi_{2, *})(y_{E/F}) =&\, \sum_{\phi\in \widehat{\Gamma}}y(F, \phi- \psi_{2}(\phi))e_\phi\\
		=& \,\sum_{\phi \in \widehat{\Gamma}_R}e_\phi  + \sum_{\phi \in \widehat{\Gamma}_U}\phi(\sigma)^{-1}e_\phi\\
		=& (1-e_I) + \sigma^{-1}e_I.\end{align*}
	The last equality here is true since $\sum_{\phi \in \widehat{\Gamma}_U}e_\phi = e_I$ and $\phi(\sigma)e_\phi = \sigma e_\phi$ for all $\phi \in \widehat{\Gamma}_U$. This proves claim (i). 
	
	Turning to claim (ii), we first note that if $E/F$ is tamely ramified, $|I|$ is prime to $\ell$. Then, in this case, $\ZZ_\ell[I]$ is the maximal $\ZZ_\ell$-order in $\QQ_\ell[I]$, and so $e_I \in \ZZ_\ell[I]$ and $\sigma^{-1}e_I  \in \ZZ_\ell[\Gamma]$. In addition, since all irreducible character $\phi \in \widehat{\Gamma}_U$ are linear, then \eqref{first eq} and \eqref{second eq} combine to imply that
	\[(1-e_I) + \sigma^{-1}e_I = {\rm Nrd}_{\QQ_\ell[\Gamma]} ( (1-e_I) + \sigma^{-1}e_I).\]
	Now, since $(1-e_I) + \sigma^{-1}e_I$ is a unit of $\ZZ_\ell[\Gamma]$, the formula from claim (i) can therefore be combined with the exactness of \eqref{Diagram: Kseq} to deduce that $\mathfrak{c}_{E/F} = \delta_{\ZZ_\ell,\QQ_\ell,\Gamma}( (1-e_I) + \sigma^{-1}e_I)$ vanishes in $K_0(\mathbb{Z}_{\ell}[\Gamma],\mathbb{Q}_{\ell}[\Gamma])$.
	
	Finally, if $\Gamma$ is equal to $I$ and has odd order, then $\sigma^{-1} e_I = e_I$ and so the formula in claim (i) implies that 
	\[ (1 - \psi_{2, *})(y_{E/F})= (1-e_I) + \sigma^{-1}e_I  = (1-e_I) + e_I = 1.\]
	Therefore, in this case one has $\mathfrak{c}_{E/F} = 0$. This completes the proof of claim (ii). 
\end{proof}

We now define a global analogue of the above twisted unramified characteristics. 

\begin{definition}\label{idelic twisted def} Let $L/K$ be a finite Galois extension of number fields and set $G:={\rm Gal}(L/K)$. Then, the `idelic twisted unramified characteristic' of $L/K$ is the element of 
	$K_{0}(\mathbb{Z}[G], \mathbb{Q}[G])$ that is equal to the sum  
	\[ \mathfrak{c}_{L/K}:= \sum_{v \in \mathcal{W}_{L/K}} {\rm i}^{G, *}_{G_w, \QQ_\ell}(\mathfrak{c}_{L_w/K_v}), \]
	where $\mathcal{W}_{L/K}$ denotes the (finite) set of non-Archimedean places $v$ of $K$ that ramify wildly in $L/K$, and $\ell$ denotes the residue characteristics of $v$. Here the element ${\rm i}^{G, *}_{G_w, \QQ_\ell}(\mathfrak{c}_{L_w/K_v})$ is considered as an element of $K_{0}(\mathbb{Z}[G], \mathbb{Q}[G])$ via the inclusion $K_{0}(\mathbb{Z}_\ell[G], \mathbb{Q}_\ell[G]) \subset K_{0}(\mathbb{Z}[G], \mathbb{Q}[G])$ induced by \eqref{Eq: KT iso K0} for each prime $\ell$.
\end{definition}

\begin{remark}\label{Rem: idelic c l-comp}\
	
	\noindent{}(i) The formal definition of $\mathfrak{c}_{L/K}$ given above differs from that given (for odd degree extensions) in \cite[(8.3)]{BBH} since the latter sums the terms ${\rm i}^{G, *}_{G_w, \QQ_\ell}(\mathfrak{c}_{L_w/K_v})$ over all non-Archimedean places of $v$. However, Lemma ~\ref{Prop: local c tame vanishes}(ii) implies that these definitions do, in fact, agree. 
	
	\noindent{}(ii) For each rational prime $\ell$, one has 
	\begin{equation*}\label{Eq: global unrami element}
		j_{\ell, \ast}(\mathfrak{c}_{L/K}) = \begin{cases} 0 , &\text{ if $\ell \notin \mathcal{W}_{L/K}^{\mathbb{Q}}$};\\
			\sum_{v|\ell}	{\rm i}^{G, *}_{G_w, \mathbb{Q}_\ell}	\mathfrak{c}_{L_w/K_v}, &\text{ if 
				$\ell \in \mathcal{W}_{L/K}^{\mathbb{Q}}$,}\end{cases}
	\end{equation*}
	where $\mathcal{W}_{L/K}^{\mathbb{Q}}$ denotes the set of residue characteristics of places of $K$ that ramify 
	wildly in $L/K$. 
\end{remark} 

\subsubsection{Interpretation of Conjecture~\ref{bbh conj}}
Firstly, we note that, since Conjecture~\ref{bbh conj} only considers the case that $L/K$ has odd degree, we know from Proposition \ref{frak a global independence}(ii) that the element $\mathfrak{a}_{L/K}$ introduced in Definition \ref{Def: frak a global} agrees with the corresponding element defined in \cite{BBH}.  

\begin{remark} It is known that this conjecture provides the first concrete link
between the theory of the square root of the inverse different and the general framework of
Tamagawa number conjectures that originated with Bloch and Kato in \cite{bk} (see the introduction to \cite{BBH}). This link, however, will play no significant role in the results that we present in this article. \end{remark}

The strongest theoretical evidence that Bley, Burns and Hahn offer in support of Conjecture~\ref{bbh conj} is recalled in the following result. 

\begin{prop}[{\cite[Cor. 8.4]{BBH}}]\label{Th: BBH thm}
	Let $L/K$ be a weakly ramified finite Galois extension of number fields of odd degree. 
	
	\begin{itemize}
		\item[(i)] Conjecture~\ref{bbh conj} is valid modulo the torsion subgroup ${\rm DT}(\ZZ[G])$ of $K_0(\ZZ[G],\QQ[G])$. 
		
		\item[(ii)] Conjecture~\ref{bbh conj} is valid provided that every place $v$ of $K$ that is wildly ramified in $L$ has the following three properties:
		\begin{itemize}
			\item[(a)] the decomposition subgroup in $G$ of any place of $L$ above $v$ is abelian;
			\item[(b)] the inertia subgroup in $G$ of any place of $L$ above $v$ is cyclic, and
			\item[(c)] the completion of $K$ at $v$ is absolutely unramified.
		\end{itemize}
	\end{itemize}
\end{prop}

\begin{remark}\label{BBH interpret remark}\
	
\noindent{}(i) Since $\mathfrak{c}_{L/K}$ belongs to $K_0(\ZZ[G],\QQ[G])$ and has finite order, the meaning of claim (i) of Proposition \ref{Th: BBH thm} is that $\mathfrak{a}_{L/K}$ belongs to ${\rm DT}(\ZZ[G])$. The result is proved by a reduction to tamely ramified extensions (see \cite[\S8B1]{BBH} for its proof) that depends on the functorial properties of the elements $\mathfrak{a}_{L/K}$ and $\mathfrak{c}_{L/K}$ (see \cite[Th. 6.1 and Rem. 8.9]{BBH}). 
	
\noindent{}(ii) Claim (ii) of Proposition \ref{Th: BBH thm} is proved by adapting methods developed by Pickett and Vinatier in \cite{PV} (which we will follow suit for the proof Lemma~\ref{Lemma: local weakly rami main} in \S\ref{S: weak rami K_0 result}).  
	
\noindent{}(iii) Aside from the result of Proposition \ref{Th: BBH thm}, the only other evidence in support of Conjecture~\ref{bbh conj} are (1) the numerical computations described in \cite[\S10]{BBH} that verify the conjecture for all Galois extensions of $\QQ$ of degree $27$ (in Theorem 10.2) and for a certain family of Galois extensions of $\QQ$ of degree $63$ (in Theorem 10.5); (2) the computation by the present author in \cite{Ku} of explicit upper bounds on the order of the elements $ \mathfrak{a}_{L/K}$ and $\mathfrak{c}_{L/K}$ in the case that $G$ is a $p$-group and $p$ is unramified in $K$ where $p$ is an odd prime.
\end{remark}

\section{Results in the relative $K$-group}
In this section, we prove Theorem~\ref{last intro thm}(i), which extends Proposition~\ref{Th: BBH thm} to an extension of arbitrary degree that the square root of the inverse different exists and that no wildly ramified place is 2-adic. 

%Firstly, we restate our main result.

%\begin{thm}\label{Thm: main}
%Let $L/K$ be a weakly ramified finite Galois extension of number fields and set $G:=\mathrm{Gal}(L/K)$.
%Suppose that the inverse difference of $L/K$ is a square and that every place $v$ of $K$ that is wildly ramified in $L$ satisfies the conditions in Proposition~\ref{Th: BBH thm}(ii) and no such $v$ is $2$-adic.
%Then, in $K_0(\mathbb{Z}[G], \mathbb{Q}[G])$, one has
%
%\[ \mathfrak{a}_{L/K} = \mathfrak{c}_{L/K}. \]
%
%\end{thm}

In view of the decomposition result of $\mathfrak{a}_{L/K}$ in Proposition~\ref{Lemma: frak a local decomp} and Definition~\ref{idelic twisted def} of $\mathfrak{c}_{L/K}$, it is enough for us to show that the stated equality holds for each place $v$ of $K$ that ramifies in $L$ (and satisfies the stated hypotheses).

%The hypotheses of the following results Theorem~\ref{Th: tame local main} and Lemma~\ref{Lemma: local weakly rami main} remove the odd degree condition from \cite[Th. 8.1]{BBH} of Bley, Burns and Hahn while assuming that the square root exists. 

\subsection{Tamely ramified extensions}
In this section, we will show that, for a tamely ramified extension of local fields, the canonical relative element defined in \S~\ref{def local element} vanishes in the relative $K$-group whenever it exists.

\subsubsection{Statement of the results}
\begin{thm}\label{Th: tame local main}
	Let $E/F$ be a finite Galois extension of $p$-adic fields and set $\Gamma := {\rm Gal}(E/F)$. If $E/F$ is tamely ramified and $\mathcal{A}_{E/F}$ exists, then in $K_0(\ZZ_p[\Gamma],\QQ_p[\Gamma])$, one has $\mathfrak{a}_{E/F} =\mathfrak{c}_{E/F}=0.$
\end{thm}

The second of the required equalities follows directly from the fact that, under the hypotheses, $\mathfrak{c}_{E/F}$ vanishes (see Remark~\ref{Cor: odd inertia} and Lemma~\ref{Prop: local c tame vanishes}).

For the next result, we fix an embedding $j^c_{p}:\mathbb{Q}^{c}\rightarrow \mathbb{Q}^{c}_{p}$. We also write $j_p^c$ for induced embedding $\zeta(\QQ^c[G])^\times \rightarrow \zeta(\QQ^c_p[G])^\times$. We let $\mathbb{Q}_{p}^{t}$ denote the maximal tamely ramified extension of $\mathbb{Q}_{p}$ and $\mathcal{O}_{p}^{t}$ denote the valuation ring of $\mathbb{Q}_{p}^{t}$. 

Recall the notations $\widehat{\Gamma}$ ($R_{\Gamma}$) and $\widehat{\Gamma}_p$ ($R_{\Gamma , p}$) from \S\ref{S: notation}. We note that $j^c_p$ also gives rise to a bijection $R_\Gamma \rightarrow R_{\Gamma, p}$ by sending $\chi$ to the function $\chi^{j}$ that is defined by setting $\chi^{j}(g) := {j_p^c}(\chi(g))$ for all $g\in \Gamma$. In this way, for each $\phi \in R_{\Gamma, p}$, with $\phi = \chi^{j}$ for some $\chi \in R_\Gamma$, the virtual character $\chi$ is equal to the function $\phi^{j^{-1}}$ that is defined by setting $\phi^{j^{-1}} (g) := j_\ell^{-1} (\phi (g))$ and hence, we write
\[{j_p^c(\tau_{E/F}')}=\sum_{\phi \in \widehat{\Gamma}_p} e_\phi \cdot (j_p^c(\tau'(F, \phi^{j^{-1}}))) \in \zeta(\mathbb{Q}_p^c[\Gamma])^{\times}. \]
Here $\tau'_{E/F}$ is the local analogy to the element defined in Definition~\ref{Def: equiv global GGS}.

The following result extends the result \cite[Lem. 8.4]{E} of Erez for tamely ramified extensions of odd degree, which uses results of Fröhlich and Taylor (see \cite[Chap. III \& IV]{F83}) for Fröhlich's conjecture. We shall follow their steps in proving it.
\begin{lemma}\label{Lem: tame local norm main}
	Fix an element $a \in E$ such that $\mathcal{A}_{E/F}=\mathcal{O}_{F}[\Gamma] \cdot a$, then there exists an element $u\in \mathcal{O}_{p}^{t}[\Gamma]^{\times}$ such that, for all  $\chi \in \widehat{\Gamma}_p$, 
	\begin{equation}\label{Eq: tame local norm main eq}
		\frac{\mathcal{N}_{F/\mathbb{Q}_{p}}(a|\chi) }{j_{p}^c(\tau'(F, \psi_{2}(\chi^{j^{-1}})-\chi^{j^{-1}}))}=\mathrm{Det_{\chi}}(u).
	\end{equation}
\end{lemma}

Before we prove this result, we shall first show that Lemma~\ref{Lem: tame local norm main} implies Theorem~\ref{Th: tame local main}.

By Taylor's fixed point theorem (see \S\ref{S: taylor fixed point}), in order to prove Theorem~\ref{Th: tame local main}, it is enough for us to show that the image in $K_0(\mathcal{O}_p^t[\Gamma], \mathbb{Q}^c_p[\Gamma])$ of $\mathfrak{a}_{E/F}$ vanishes. Taking account of Remark~\ref{frak a' def}, it is therefore sufficient for us to show that the image in $K_0(\mathcal{O}_p^t[\Gamma], \mathbb{Q}^c_p[\Gamma])$ of $\mathfrak{a}'_{E/F}$ vanishes . 

%In the sequel, we set $\delta_{\Gamma, p} := \delta_{\ZZ_p,\QQ_p^c,\Gamma}$, and 
%
%\[ \mathfrak{a}'_{E/F}=\Delta(\mathcal{A}_{E/F}) - \delta_{\Gamma, p}(j^c_{p}(T^{(2)}_{E/F})) \]
%
%Then, Definition~\ref{local a def} implies that $\mathfrak{a}'_{E/F} = \mathfrak{a}- U_{E/F}$. Recalling  that the image of $U_{E/F}$ vanishes in $K_0(\mathcal{O}_p^t[\Gamma], \mathbb{Q}^c_p[\Gamma])$ (by Proposition~\ref{Prop: decomp pre prop}(iii)), it is therefore enough for us to show that the image of $\mathfrak{a}'_{E/F}$ vanishes in $K_0(\mathcal{O}_p^t[\Gamma], \mathbb{Q}^c_p[\Gamma])$.

To do this, we fix an element $a$ of $E$ such that $\mathcal{A}_{E/F}= \mathcal{O}_{F}[\Gamma]\cdot a$, and set $\delta_{\Gamma, p} := \delta_{\ZZ_p,\QQ_p^c,\Gamma}$ and $v:=\delta_{F}/ j_p^c(\tau_{F})$. Then, by the explicit formula in Proposition~\ref{Prop: local in Q}(ii), one has
%we can express $\Delta(\mathcal{A}_{E/F})$ as an image under $\delta_{\ZZ_p,\QQ_p^c,\Gamma}$, and hence rewrite the equality above
\begin{align*}
	\mathfrak{a}'_{E/F} =&\, \delta_{\Gamma, p} \left((j_p^c(T_{E/F}^{(2)} ))^{-1} \cdot \sum_{\chi \in \widehat{\Gamma}_p} (\delta_{F}^{\chi(1)} \cdot \mathcal{N}_{F/\mathbb{Q}_{p}}(a|\chi))e_{\chi}   \right)\\
	=&\, \delta_{\Gamma, p} \left(\sum_{\chi \in \widehat{\Gamma}_p} \left( (\delta_{F}/ j_p^c(\tau_{F}))^{\chi(1)} \cdot \frac{\mathcal{N}_{F/\mathbb{Q}_{p}}(a|\chi) }{j_{p}^c(\tau'(F, \psi_{2}(\chi^{j^{-1}})-\chi^{j^{-1}}))}  \right) e_{\chi}  \right)\\
	=&\, \delta_{\Gamma, p}\bigr(x_1\bigr) + \delta_{\Gamma, p} \left(\sum_{\chi \in \widehat{\Gamma}_p} \frac{\mathcal{N}_{F/\mathbb{Q}_{p}}(a|\chi) }{j_{p}^c(\tau'(F, \psi_{2}(\chi^{j^{-1}})-\chi^{j^{-1}}))} e_{\chi}  \right) ,
\end{align*}
%
%note here $\tau_{K}^{\Gamma}:=\sum_{\chi \in \widehat{\Gamma}_p} j_p^c(\tau_{F})^{\chi(1)}$ directly from the definition that $\tau_F^\Gamma=:{\rm Nrd}_{\QQ^c[\Gamma]}(\tau_{F})$ (that is $\chi(1)$ is the exponent)
with $x_1:= \sum_{\chi \in \widehat{\Gamma}_p} v^{\chi(1)} e_\chi = {\rm Nrd}_{\QQ_p^c[\Gamma]}(v)$. Since $v$ is a unit of $\mathcal{O}_p^t$ (see Proposition~\ref{Prop: local in Q}(iii)), the image in $K_0(\mathcal{O}_p^t[\Gamma], \mathbb{Q}^c_p[\Gamma])$ of the element $\delta_{\Gamma, p}(x_1)$ vanishes.

Now, assuming Lemma~\ref{Lem: tame local norm main}, we have $x_2= \sum_{\chi \in \widehat{\Gamma}_p}  x_{2, \chi} e_\chi$ such that, for all $\chi \in \widehat{\Gamma}_p$,
\[ x_{2, \chi}:= \frac{\mathcal{N}_{F/\mathbb{Q}_{p}}(a|\chi) }{j_{p}^c(\tau'(F, \psi_{2}(\chi^{j^{-1}})-\chi^{j^{-1}}))}= {\rm Det}_{\chi}(u) ,\]
where $u$ belongs to $\mathcal{O}_p^t[\Gamma]^\times$. Therefore the image in $K_0(\mathcal{O}_p^t[\Gamma], \mathbb{Q}^c_p[\Gamma])$ of the element $\delta_{\Gamma, p}(x_2)$ vanishes as a consequence of the fact that the kernel of the composite homomorphism \eqref{Eq: Taylor kernel Det} is ${\rm Det}(\mathcal{O}_p^t[\Gamma]^\times)$.
%the composition \eqref{Eq: Taylor kernel Det} implies that $\delta_{\Gamma, p}(x)$ also vanishes in $K_0(\mathcal{O}_p^t[\Gamma], \mathbb{Q}^c_p[\Gamma])$. 
This completes the proof of Theorem~\ref{Th: tame local main}.

\subsubsection{Reduction to inertia subgroups}
In this section, we recall the fractional ideals in $\mathbb{Q}^c_p$ generated by norm resolvents and Galois Gauss sums from \cite[Chap. III]{F83} and the functorial behaviour of both elements. We let $U(\mathbb{Q}_p^c)$ denote the group of units of the ring of integer $\mathbb{Z}^c_p$ of $\mathbb{Q}_p^c$.

Fix an element $a\in E$ such that it generates a normal basis of $E/F$. Then, for each $\chi \in R_{\Gamma, p}$, we write $P(E/F, \chi)$ for the class of $(a|\chi)$ modulo $U(\mathbb{Q}_{p}^{c})$ (that is, the fraction ideal in $\mathbb{Q}^c_p$ generated by $(a|\chi)$), and we note that $P(E/F, \chi)$ is independent of the choice of $a$ and is inflation invariant (cf. \cite[pp. 106-107]{F83}). We also define the element $\mathcal{N}_{F/\mathbb{Q}_{p}}P(E/F, \chi)= \prod_{\omega}P(E/F, \chi^{\omega^{-1}})^{\omega}$, where $\{\omega\}$ is a transversal of $\Omega_{F}$ in $\Omega_{\mathbb{Q}_{p}}$. One can see (directly from the definition) that this element is the fractional ideal in $\mathbb{Q}_{p}^{c}$ generated by $\mathcal{N}_{F/\mathbb{Q}_{p}}(a|\chi)$ (with $a$ as above), and we note that it is independent of the choice of $\omega$ (cf. \cite[Chap. I, \S4, Prop. 4.4]{F83}). 

In the sequel, we abbreviate $\mathcal{N}_{F/\mathbb{Q}_{p}}P(E/F, \chi)$ to $\mathcal{N}_{F/\mathbb{Q}_{p}}P(\chi)$. For $\phi \in R_\Gamma$, we let $(j_{p}^c(\tau(F, \phi)))$ denote the class of $j_{p}^c(\tau(F, \phi))$ modulo $U(\mathbb{Q}_{p}^{c})$. We also write $\Delta$ for the inertia subgroup $\Gamma_{0}$ of $\Gamma$ and $B:=E^{\Delta}$ (so that $E/B$ is of odd degree by Remark~\ref{Cor: odd inertia}(i) and $B/F$ is unramified). We set $\mathrm{res}:=\mathrm{res}^{\Gamma}_{\Delta}$. 

The following results rely heavily on the methods and result of \cite[Th. 5.2]{E} and arguments in \cite[Chap. III, \S5-7]{F83}.
\begin{prop}\label{Prop: local ideal}\
	\begin{itemize}
		\item[(i)] For all $\chi \in R_{\Gamma}$, one has
		$(\tau(B, \mathrm{res} \chi))= (\tau(F, \chi))^{[B:F]}$ for ideals.
		\item[(ii)] For all $\chi \in R_{\Gamma, p}$, the following claims are valid.
		\begin{itemize} 
			\item[(a)] Fix an element $a\in E$ such that $\mathcal{A}_{E/F}=\mathcal{O}_{F}[\Gamma] \cdot a $ and an element $b\in E$ such that $\mathcal{A}_{E/F}=\mathcal{O}_{B}[\Delta] \cdot b $. Then, one has $ (b|\mathrm{res}\chi)_{E/B}= (a|\chi)_{E/F} \cdot\mathrm{Det}_{\chi}(\lambda)$, where $\lambda\in \mathcal{O}_{B}[\Gamma]^{\times}$.	
			\item[(b)] $\mathcal{N}_{B/\mathbb{Q}_{p}}P(\mathrm{res} \chi)  = \mathcal{N}_{F/\mathbb{Q}_{p}}P(\chi)^{[B:F]}$.
			\item[(c)] $\mathcal{N}_{F/\mathbb{Q}_{p}}P(\chi)=({j^c_{p}(\tau(F, \psi_{2}(\chi^{j^{-1}})-\chi^{j^{-1}}))})$ .
		\end{itemize}
	\end{itemize}
\end{prop}

%Before we can show Proposition~\ref{Prop: local ideal}, we need to recall the result on the similar functorial behaviour of resolvents and Galois Gauss sums. 
%We are therefore able to reduce to the totally and tamely ramified extension, which, under our hypotheses is of odd degree. We can then adapt Erez's computation in \cite[\S7]{E}. 

%To prove this result, we shall follow the steps for \cite[Th. 5.2]{E}, which use arguments in \cite[Chap. III, \S5-7]{F83}.

\begin{proof}
	Claim (i) is given in \cite[Cor. 1 to Th. 25]{F83}. 
	
	Now we move to claim (ii). In the rest of this proof, we shall consider $\chi \in R_{\Gamma, p}$.
	
	For part (a), we recall the set $\mathrm{Map}(\Gamma, E)^{\Delta}$ of maps of $\Delta$-sets $\Gamma \rightarrow E$ 
	%this means, $\Gamma$ and $E$ are $\Delta$-sets and $f$ is an equivariant map
	(more precisely, the set of maps $f:\Gamma \rightarrow E$ such that $h(f(g))= f(h g)$ for all $h\in \Delta$ and $g\in \Gamma$). 
	%ie. $\Delta$ act on $\Gamma$ by left multiplication
	It is a $B$-algebra via $(xf)(g)=xf(g)$ for all $x \in B$, $f \in \mathrm{Map}(\Gamma, E)^{\Delta}$, $g \in \Gamma$, and it is a $\Gamma$-module via $(gf)(g')=f(g'g)$ for all $f \in \mathrm{Map}(\Gamma, E)^{\Delta}$, $g$, $g' \in \Gamma$. Then, there is an isomorphism of $B[\Gamma]$-modules (taken from \cite[Chap. III, \S6, Lem. 6.2]{F83})
	\[ j:B \otimes_{F}E \rightarrow \mathrm{Map}(\Gamma, E)^{\Delta}\]
	defined by $(j(x\otimes y))(g)=xg(y)$ for $x \in B$ , $y \in E$, and $g \in \Gamma$, and it induces an isomorphism of $\mathcal{O}_{B}[\Gamma]$-modules	
	\[ \mathcal{O}_{B}\otimes_{\mathcal{O}_{F}}\mathcal{A}_{E/F} \rightarrow \mathrm{Map}(\Gamma , \mathcal{A}_{E/F})^{\Delta}. \]
	(See, for example, the argument of \cite[Lem. 2.4]{CV} and the isomorphism given in (8) of loc.cit.
	%, the composition $\mathcal{O}_{B}\otimes_{\mathcal{O}_{F}}\mathcal{A}_{E/F} \xrightarrow{h} \mathrm{Map}(\Gamma , \mathcal{A}_{E/F})^{\Delta} \xrightarrow{\cong} \ZZ[\Gamma] \otimes_{\ZZ[\Delta]} \mathcal{A}_{E/F}$ is an isomorphism 
	with notations $E=K$, $B=F$, $F=k$, and $\mathcal{A}_{E/F}$ represented by $\mathcal{P}_{K}^{-n}$.)
	
	We also construct a map $h_b:\Gamma\rightarrow E$ associated to $b$ by setting (as in \cite[Chap. III, \S6, (6.7)]{F83}) 
	\[ h_b(g) = \left\{ \begin{array}{ll}
		g(b), & \quad \mbox{if } \quad g \in \Delta ;\\
		0, & \quad \mbox{if } \quad g\in \Gamma \setminus \Delta .\end{array} \right. \] 
	We note that $h_{b}$ belongs to $\mathrm{Map}(\Gamma , \mathcal{A}_{E/F})^{\Delta}$ and it generates $\mathrm{Map}(\Gamma , \mathcal{A}_{E/F})^{\Delta}$ freely over $\mathcal{O}_{B}[\Gamma]$ (this can be shown by the argument given in loc.cit. replacing $\mathcal{O}_{L}$ by $\mathcal{A}_{E/F}$).
	
	In this way, $j^{-1}(h_{b})$ is a free generator of $\mathcal{O}_{B}\otimes_{\mathcal{O}_{F}}\mathcal{A}_{E/F}$ over $\mathcal{O}_{B}[\Gamma]$. In addition, $1\otimes a$ is also a free generator of $ \mathcal{O}_{B}\otimes_{\mathcal{O}_{F}}\mathcal{A}_{E/F}$ over $\mathcal{O}_{B}[\Gamma]$. Then, there exists an element $\lambda\in \mathcal{O}_{B}[\Gamma]^{\times}$ such that $j^{-1}(h_b)={\lambda}(1\otimes a)$. The rest of the proof of part (a) is just copying down \cite[Chap. III, \S6, (6.9)-(6.12)]{F83}, with $j^{-1}g$ in loc.cit. replaced by $j^{-1}(h_b)$. 
	
	With the result in part (a), the proof for part (b) is exactly the same as the bottom of p.131 of loc.cit. (see \cite[Chap. I, \S4, Prop. 4.4]{F83} for the Galois action formula on resolvents).
	
	Now, given the results in claim (i) and (ii)(b), we are therefore reduced to showing that the stated equality in part (c) is valid for $E/B$ (that is of odd degree), and this is proved by Erez in \cite[Th. 5.2 and \S7.1]{E}. Then, one has
	%for part (c), we recall from Remark~\ref{Cor: odd inertia}(i) that, under our hypotheses, $E/E_0$ is of odd degree. Then, with the results from claim (i) and (ii)(b), one has 
	%
	\begin{align*}
		(j_p^c(\tau(F, \psi_{2}(\chi^{j^{-1}}) - \chi^{j^{-1}})))^{[B:F]}
		=&\, (j_p^c(\tau(B, \mathrm{res} (\psi_{2} (\chi^{j^{-1}}) -\chi^{j^{-1}}))))\\
		=&\, (j_p^c(\tau(B, \psi_{2}( \mathrm{res} \chi^{j^{-1}})) - \mathrm{res} \chi^{j^{-1}}))\\
		=&\, \mathcal{N}_{B/\mathbb{Q}_{p}}P(\mathrm{res} \chi)  
		= \mathcal{N}_{F/\mathbb{Q}_{p}}P(\chi)^{[B:F]}.
	\end{align*}
	Here, the second equality follows from Proposition~\ref{Prop: adam}(ii), and the third equality is given by the result of Erez. 
	%Now, the third equality is true by Erez's results for extensions of odd degrees \cite[Th. 5.2]{E}. 
	This finishes the proof of our claim.
	
	%we note that by claim (i) and (ii)(b), both sides of the equality in (c) have similar functorial behaviour, therefore it suffices to show that the equality holds for $E/E_0$. This is true
	%We recall from Remark~\ref{Cor: odd inertia}(i) that, under our hypotheses, $E/E_0$ is of odd degree, hence we can apply Erez's results for extensions of odd degrees \cite[Th. 5.2]{E}.
	%By additivity, it is enough to show that the equality holds for all irreducible characters of $\Gamma_{0}$, and this is done by Erez in \cite[\S7.1]{E}. (Alternatively, since $E/E_0$ is an odd degree extension, one can apply Theorem 5.2 in loc.cit. to conclude the equality.) 
	%Hence we have
	%
	
\end{proof}

%In the sequel, we omit $F$ in $\tau(F, \chi)$ and write $\tau(\chi)$ for the Galois Gauss sums.

\subsubsection{Proof of Lemma~\ref{Lem: tame local norm main}}
To prove Lemma~\ref{Lem: tame local norm main}, as Erez pointed out in \cite[Lem. 8.4]{E} for extensions of the odd degree, one needs only to copy the argument for \cite[Ch. IV, \S1, Th. 31]{F83} upon replacing Theorem 23 and 25 by Proposition~\ref{Prop: local ideal}(ii)(c) and (ii)(a) respectively. For a complete proof we refer the reader to the PhD thesis \cite[\S6.1.4]{K_phd} of the author.

We also note that, in this proof we are using a modified local Galois Gauss sums $\tau^{*}(F, \phi)=\tau(F, \phi)y(F, \phi)^{-1}z(F, \phi)$ for $\phi \in R_\Gamma$ (defined as in \cite[Ch. IV, \S1, (1.6)]{F83}) in place of $\tau'(F, \phi)$. And we recall that, for all $\phi \in R_\Gamma$, one has $z(F, \phi)=\det_{\phi}(\gamma')$ for some $\gamma' \in \Gamma$ and that the map $\phi \rightarrow z(F, \phi)$ lies in ${\rm Hom}^+(R_\Gamma, \mu(\mathbb{Q}^c))^{\Omega_{\mathbb{Q}}}$ (where $\mu(\mathbb{Q}^c)$ denotes the group of roots of unity in $\mathbb{Q}^c$), and so $z(F, \psi_{2}(\phi) - \phi)=\det_{\phi}(\gamma')$ (by Proposition~\ref{Prop: adam}(iii)).

\subsection{Special cases of weakly ramified extensions}\label{S: weak rami K_0 result}
%The hypothesis of the following result remove the odd degree condition from \cite[Th. 8.1]{BBH} of Bley, Burns and Hahn while assuming that the square root exists. 

The proof of the following result is exactly the same to \cite[Th. 8.1]{BBH} of Bley, Burns and Hahn.

\begin{lemma}\label{Lemma: local weakly rami main}
Fix an odd prime number $p$. Let $E/F$ be an abelian and weakly and wildly ramified Galois extension of $\mathbb{Q}_p$ such that $F/\mathbb{Q}_p$ is unramified and write $\Gamma:=\mathrm{Gal}(E/F)$. Suppose $\mathcal{A}_{E/F}$ exists and that the inertia subgroup of $\Gamma$ is cyclic. Then, in $K_0(\ZZ_p[\Gamma],\QQ_p[\Gamma])$, one has $\mathfrak{a}_{E/F} = \mathfrak{c}_{E/F}$. 
\end{lemma}

We note that, under these hypotheses, the computation by Bley and Cobbe in \cite[\S5]{BC} and those of Pickett and Vinatier in \cite[\S3.2]{PV} do not rely on the assumption that $\Gamma$ is of odd degree but that $\Gamma$ is abelian and that the inertia subgroup $\Gamma_{0}$ of $\Gamma$ has odd order $p$ (see \cite[Rem. 1 and \S3.1]{BC} and \cite[Cor. 3.4]{PV}). The reader can find a complete proof of this result in \cite[\S7.1]{K_phd}.

\section{Elements in the class group}
%Inspired by Erez's conjecture, in this section, we study the image in the class group of the global canonical relative elements defined in \S\ref{S: elements of BBH} under the surjective connecting homomorphism in \eqref{Diagram: Kseq}. 

Let $L/K$ be any weakly ramified finite Galois extension of number fields and set $G:=\mathrm{Gal}(L/K)$. 
Suppose $\mathcal{A}_{L/K}$ exists, then it is projective (see Proposition~\ref{Lemma: weakly rami proj iff}). 
In view of the `classical' Galois module theory, in this section, we consider the stable-isomorphism class $[\mathcal{A}_{L/K}]$ defined by $\mathcal{A}_{L/K}$ in the class group ${\rm Cl}(\mathbb{Z}[G])$.

\subsection{Symplectic Galois-Jacobi sums}\label{S: Symp J}
%To start, we recall from \S~\ref{S: erez conj} that, in the case of tamely ramified extensions, Erez's conjecture is equivalent to showing that $[\mathcal{O}_{L}] = [\mathcal{A}_{L/K}]$ whenever $\mathcal{A}_{L/K}$ exists.

%Caputo and Vinater have in \cite[Th. 2]{CV} proved the equality by studying certain torsion modules under the hypothesis that $L/K$ is locally abelian (ie. the decomposition group in $G$ of every ramified place of $L/K$ is abelian). 
%\begin{thm}[{Caputo and Vinater}]\label{Th: CV thm}
%Let $L/K$ be a tamely ramified Galois extension of number fields that is locally abelian and $\mathcal{A}_{L/K}$ exists. Then, $[\mathcal{A}_{L/K}] =[\mathcal{O}_{L}]$.
%\end{thm}

In the section, we define relevant elements in Conjecture~\ref{C:boas} of symplectic Galois-Jacobi sums.

To start, for a real number $x$, we write $\text{sign}(\chi)\in \{\pm 1 \}$ for the sign of $x$. Then, for each $\chi \in \widehat{G}$ we write
\begin{equation}\label{Eq: def symp JG}
J'_{2, S} (L/K, \chi):=
\left\{ \begin{array}{ll}
	\mathrm{sign}\bigl(    
	\tau(K, \psi_{2}(\chi) - 2\chi) \cdot y(K, 1-\psi_{2}(\chi))\bigr) , & \quad \mbox{if } \quad \chi \in \mathrm{Symp}(G) ;\\
	1, & \quad \mbox{otherwise} .\end{array} \right.
\end{equation}
and we define the `equivariant symplectic Galois-Jacobi sum' of $L/K$ by setting
\[
J'_{2, S, L/K} := \sum_{\chi \in \widehat{G}} J'_{2, S} (L/K, \chi).
\]
Clearly, such element has order at most $2$ and belongs to $\zeta(\QQ[G])^{\times}$.
%
%We remark here that the elements above are defined using the algebraic $K$-theory technique (for a definition using the classical $Hom$-description, see \cite[\S 8]{AC}). In particular, the definition is not restricted to the case of tamely ramified extension.
%
%
%\begin{thm}[{Agboola, Burns, Caputo and Kuang}]\label{Th: AC thm}
%For each tamely ramified $L/K$ such that $\mathcal{A}_{L/K}$ exists, in the class group, the equality  $[\mathcal{A}_{L/K}] = [\mathcal{O}_{L}] + \mathcal{J}_{2, S, L/K}$ holds.
%\end{thm}
%
To interpret this term on signs, we show the following results (cf. \cite[Th. 8.1 and Th. 1.7]{AC}). 

In the sequel, for any integer $m \geq 1$, we write $H_{4m}$ for the generalised quaternion group of order $4m$.
\begin{lemma}\label{Lemma: sign depend}
Suppose that $L/K$ is tamely ramified and that $\mathcal{A}_{L/K}$ exists. The following claims are valid.
\begin{itemize}
\item[(i)] The element $J'_{2, S, L/K}$ depends solely on the value of $\psi_{2}$-twisted unramified characteristics of non-trivial irreducible symplectic characters of $G$. 
\item[(ii)] Suppose further that $K$ is an imaginary quadratic field such that $\Cl(\mathcal{O}_{F})$ contains an element of order $4$ and that $G:=\Gal(L/K)\cong H_{4\ell}$ for some sufficiently large prime $\ell$ with $\ell \equiv 3 \pmod{4}$. Then, there exists infinitely many tamely ramified extension $L/K$ such that $J'_{2, S, L/K}$ is not trivial.
\end{itemize}
\end{lemma}

In light of the decomposition result in Proposition \ref{Prop: JG decomp}, it is essential for us to recall the some local results. Claim (ii) of the following results motivate the definitions above.  
\begin{prop}\label{Lem: local root number + thm}
Let $E/F$ be a tamely ramified Galois extension of non-archimedean local fields that has odd ramification degree and set $\Gamma:= {\rm Gal}(E/F)$. 
\begin{itemize}
\item[(i)]  (\cite[Th. 7.1]{AC}) For each $\phi \in R_\Gamma$, we let $W(F, \phi)$ denote the local root number of $\phi$ as discussed in \cite[Chap. II, \S4]{M}.
Then, for each $\chi$ in ${\rm Symp}(\Gamma)$, one has the root number $W(F,\psi_2(\chi)) = W(F,2\chi) = 1$. In particular, this implies that 
\[
\mathrm{sign}(\tau(F,\psi_{2}(\chi))) =  \mathrm{sign}(\tau(F,2\chi))=1.
\]
\item[(ii)] (\cite[Prop. 8.2]{AC})  If $\Gamma \cong H_{4m}$ with odd $m$, then there exists an irreducible symplectic character $\chi$ of $\Gamma$ such that 
\[ y(F, \psi_{2}(\chi))=-1. \]
\end{itemize}
\end{prop}

\begin{remark}\
\begin{itemize}
\item[(i)] Claim (ii) of the above result implies that, for a local extension $E/F$, the map $\chi \mapsto y(F, \psi_{2}(\chi))$ no longer lies in ${\rm Hom}^{+}(R_{\Gamma}, \mu(\QQ^c))^{\Omega_{\mathbb{Q}}}$. This shows that the map $\chi \mapsto y(F, \psi_{2}(\chi))$ behaves differently comparing to the case of extensions of odd degree 
%(for which ${\rm Hom}^{+}(R_{\Gamma}, \mu(\QQ^c))^{\Omega_{\mathbb{Q}}}={\rm Hom}(R_{\Gamma}, \mu(\QQ^c))^{\Omega_{\mathbb{Q}}}$), 
and also to the map $\chi \mapsto y(F, \chi)$. 
\item[(ii)] For weakly and wildly ramified cases, it remains unknown that whether the second Galois-Jacobi sums attached to irreducible symplectic characters of $G$ are strictly positive real numbers. In fact, if one removes any of the conditions in Proposition \ref{Lem: local root number + thm}, then the computation of root numbers of $\psi_{2}$-twisted irreducible symplectic characters can become much more difficult due to a technique step. For an explicit example of a certain class of meta-cyclic groups, we refer the reader to the author's PhD thesis \cite[\S 6.3]{K_phd}.
\end{itemize}	
\end{remark}

Proposition \ref{Lem: local root number + thm}(i) and \eqref{Eq: def symp JG} immediately imply claim (i) of Lemma \ref{Lemma: sign depend}.

Under the hypothesis of Lemma \ref{Lemma: sign depend}(ii), in \cite[Th. 1.7 and \S10]{CV} the authors have shown that there exists an extension $L/K$ such that (i) $\mathcal{A}_{L/K}$ exists; (ii) $L/K$ is ramified at a single place $v'$  and the decomposition subgroup in $G$ of any place of $L$ above $v'$ is isomorphic to $G$. Therefore, there exists an unique $\chi \in {\rm Symp}(G)$ such that
\[ y(K, \chi) = \prod_{v|d_L} y(K_v, \chi_v) =  y(K_{v'}, \chi_{v'}) =y(K_{v'}, \chi) = -1  ,\]
where the last equality follows from Proposition \ref{Lem: local root number + thm}(ii) above. Thus, this computation combines with \eqref{Eq: def symp JG} to imply that $J'_{2, S, L/K}$ is not necessarily a trivial element (and has order $2$ in this case).

\begin{definition}\label{Def: cls symp JG}
	In terms of the extended boundary map of Burns and Flach (see \S\ref{S: extended boundary map}) and the homomorphism given in \eqref{Diagram: Kseq}, we define an element of $\Cl(\ZZ[G])$ by setting 
	\[
	\mathcal{J}_{2, S, L/K}:= \partial^0_{\ZZ,\QQ,G}\circ \delta_{G} ( J'_{2, S, L/K}  ).
	\]
\end{definition}

\subsection{New evidence for Conjecture~\ref{C:boas}}\label{S: weakly classgroup}
In this section, we show that the results presented in the previous sections can be combined with the central result of Bley and Cobbe in \cite{BC} to prove Conjecture~\ref{C:boas} for a special class of weakly ramified extensions, which extends \cite[Th. 1.5]{AC} and thereby provide an alternative argument for the latter result.

First, we state a series of properties the global extensions should satisfy.

\begin{hypothesis}\label{big hyp} 
	Let $L/K$ be a finite Galois extension of number fields that is at most weakly ramified and set $G:={\rm Gal}(L/K)$. We assume the following:	
	\begin{itemize}
		\item[(i)] The square root of inverse different $\mathcal{A}_{L/K}$ exists.
	\end{itemize}
	For every place $v$ of $K$ that is wildly ramified in $L$,
	\begin{itemize}	
		\item[(ii)] $v$ is lying above an odd prime number.
		\item[(iii)] The decomposition subgroup in $G$ of any place of $L$ above $v$ is abelian.
		\item[(iv)] The inertia subgroup in $G$ of any place of $L$ above $v$ is cyclic.
		\item[(v)] The completion of $K$ at $v$ is absolutely unramified.
		\item[(vi)] the inertia degree of $v$ in $L/K$ is prime to the absolute degree of the completion $K_v$. 
	\end{itemize}	
	%	
	%Moreover, for such wildly ramified $v$ with $w\in L$ lying above $v$, with $p$ the prime number lying number below $v$, 
	%
	%\begin{itemize}
	%\item[(vi)] $[K_v : \mathbb{Q}_p]$ co-primes with the inertia degree of $L_w/K_v$.
	%\end{itemize}
	%
\end{hypothesis}

Now, we can state our main results of this section.
\begin{thm}\label{Lem: local ab Erez}
Let $L/K$ be a weakly ramified finite Galois extension of number fields and set $G:={\rm Gal}(L/K)$. Then, the following claims are valid.
\begin{itemize}
\item[(i)] Suppose $L/K$ satisfies conditions (i) - (v) in Hypothesis \ref{big hyp}  and let $W_{L/K}$ denote the Cassou-Nogu\`es-Fr\"ohlich root number class. Then, in ${\rm Cl}(\mathbb{Z}[G])$, one has
\[ [\mathcal{A}_{L/K}] = W_{L/K} +\mathcal{J}_{2, S, L/K}. \]
\item[(ii)] Suppose $L/K$ satisfies all of the conditions in Hypothesis \ref{big hyp}. Then, Conjecture~\ref{C:boas} is valid for $L/K$. 
	\end{itemize}	
\end{thm}
%
%Hypotheses (i) is necessary since Proposition~\ref{Prop: local abelian inertia p-group} only covers weakly and wildly ramified extension, so even in locally abelian cases, we still need to assume that the inertia subgroup is of odd order for tamely ramified local extensions.
%

%We note that Lemma~\ref{Lem: local ab Erez}(ii)(iii) are extensions of Theorem~\ref{Th: CV thm} and \ref{Th: AC thm}.

In the sequel, we let $\mathcal{W}_{L/K}$ denote the set of non-Archimedean places $v$ of $K$ that ramify wildly in $L/K$.

Comparing part (i) of Theorem~\ref{Lem: local ab Erez} with the result of the image in the classgroup of $\mathfrak{a}_{L/K}$ in Proposition~\ref{frak a global independence}(iv), it is sufficient for us to show that, under the stated hypotheses, the difference between the class $W^{(2)}_{L/K}$ and the class $W_{L/K}$ in ${\rm Cl}(\ZZ[G])$. We shall address this in part (i) of the following result, which relies on the result \cite[Prop. 3.1]{BB} of Bley and Burns for the `equivariant global epsilon constant'.

The proof of part (ii) of the following results adapts the argument of \cite[Lem. 8.7]{BBH} of Bley, Burns and Hahn.

%Next, we need to the show following preliminary results. We again let ${\rm Symp}(G)$ denote the set of irreducible symplectic characters of $G$.
\begin{lemma}\label{Prop: local ab frak c proj}
Suppose $L/K$ satisfies conditions (i) in Hypothesis \ref{big hyp}. Then, the following claims are valid.
\begin{itemize}
\item [(i)]  $W^{(2)}_{L/K} = W_{L/K} + \mathcal{J}_{2, S, L/K}$. 	
\item[(ii)] Suppose that, in addition, $L/K$ also satisfies conditions (ii) and (iii) in Hypothesis \ref{big hyp}. Then, the image in ${\rm Cl}(\mathbb{Z}[G])$ of $\mathfrak{c}_{L/K}$ vanishes. 
\end{itemize}
\end{lemma}
\begin{proof}
We set $y_{L/K}^{(2)}:= (1 -\psi_{2, *})(y_{L/K})$ and $y_{L_w/K_v}^{(2)}:= (1- \psi_{2, *})(y_{L_w/K_v})$.	

To prove claim (i), we recall (from Proposition~\ref{frak a global independence}(iv)) the explicit expression of $W^{(2)}_{L/K}$.
\begin{align}\label{W2 proj eq 1}
	W^{(2)}_{L/K}\nonumber
	=&\, \partial^0_{\ZZ,\QQ,G}\bigl(\delta_G({\rm Nrd}_{\QQ^c[G]}(x)\cdot T^{(2)}_{L/K})\bigr) \nonumber \\ 
	=&\, \partial^0_{\ZZ,\QQ,G} \left( \delta_G \left( {\rm Nrd}_{\QQ^c[G]}(x)\cdot \bigl(\tau^{\dagger}_{L/K} \cdot J_{2, L/K}\cdot y_{L/K}^{(2)}\bigr)^{-1}\right)\right)\nonumber \\ 
	%\bigl(\delta_G({\rm Nrd}_{\QQ^c[G]}(x)\cdot (\tau^{\dagger}_{L/K} \cdot J_{2, L/K}\cdot (\psi_{2, *}-1)(y_{L/K}^{-1}))^{-1})\bigr) \nonumber \\ 
	=&\, \partial^0_{\ZZ,\QQ,G}\bigl(\delta_G({\rm Nrd}_{\QQ^c[G]}(x)\cdot (\tau^{\dagger}_{L/K})^{-1})\bigr) -  \partial^0_{\ZZ,\QQ,G}\bigl( \delta_G ( J_{2, L/K}\cdot y_{L/K}^{(2)})\bigr)
\end{align}
where second equality is given in Proposition~\ref{Prop: JG decomp}(iii).

We first claim that the second term of \eqref{W2 proj eq 1} is equal to $- \mathcal{J}_{2, S, L/K}= \mathcal{J}_{2, S, L/K}$. To see this, we note that
\begin{align*}
	\partial^0_{\ZZ,\QQ,G}\bigl( \delta_G ( J_{2, L/K}\cdot y_{L/K}^{(2)})\bigr) - \mathcal{J}_{2, S, L/K}= &\,
	\partial^0_{\ZZ,\QQ,G} \circ \delta_G ( J_{2, L/K}\cdot y_{L/K}^{(2)} \cdot    (J'_{2, S, L/K})^{-1})\\
	=&\,  \partial_{\ZZ,\QQ,G}^0 \bigl( \partial_{\ZZ,\QQ,G}^1 ({\rm Nrd}_{\QQ[G]}^{-1} (J_{2, L/K}\cdot y_{L/K}^{(2)} \cdot    (J'_{2, S, L/K})^{-1}))\bigr)\\
	=&\, 0
\end{align*}
where the second equality is deduced from \eqref{description of image nrd RR}, the explicit definition of $J'_{2, S, L/K}$ in \eqref{Eq: def symp JG} and Lemma \ref{ext bound hom}(ii). And the third equality follows from the fact that $\partial_{\ZZ,\QQ,G}^0 \circ \partial_{\ZZ,\QQ,G}^1$
is equal to the zero homomorphism (by the exactness of \eqref{Diagram: Kseq}).

Next, to prove that the first term of \eqref{W2 proj eq 1} is equal to $-W_{L/K}=W_{L/K}$, we recall the (global) $\zeta(\RR[G])$-valued `epsilon factor' function $\epsilon(s)$ of a complex variable $s$ given by Bley and Burns in \cite[\S3]{BB} with $\epsilon_{L/K}:=\epsilon(0) \in \zeta(\RR[G])^{\times}$. In particular, by Proposition 3.4 and Remark 3.5 of loc.cit., one has
\[{\rm Nrd}_{\RR[G]}(y)\cdot (\epsilon_{L/K})^{-1} ={\rm Nrd}_{\QQ^c[G]}(x)\cdot (\tau_{L/K}^{\dagger})^{-1} \in \zeta(\QQ[G])^\times, \]
with a suitable choice of $y$ (and the choice of $x$ as in the proof of Lemma~\ref{Prop: global a in Q}). Upon recalling the explicit definition of the extended boundary homomorphism on $\zeta(\RR[G])^{\times}$ from \cite[(44)]{BF}, one has that
\begin{align*}
	\partial^0_{\ZZ,\QQ,G}\bigl(\delta_G({\rm Nrd}_{\QQ^c[G]}(x)\cdot (\tau^{\dagger}_{L/K})^{-1})\bigr)=	 &\, \partial^0_{\ZZ,\QQ,G}\bigl(\delta_G({\rm Nrd}_{\RR[G]}(y)\cdot (\epsilon_{L/K})^{-1})\bigr) \\
	=&\, \partial^0_{\ZZ,\RR,G}\bigl(\partial^{1}_{\ZZ,\RR,G} (y))\bigr)  - \partial^0_{\ZZ,\RR,G}\bigl(\delta_G( \epsilon_{L/K})\bigr) \\
	=&\, -\partial^0_{\ZZ,\RR,G}\bigl(\delta_G( \epsilon_{L/K})\bigr)\\ 
	=&\, -W_{L/K}.
\end{align*}
Here, the second equality follows from Lemma~\ref{ext bound hom}(ii) (extended to $\zeta(\RR[G])^{\times}$, see Remark~\ref{Rem: extend bound extend to R}) and the right two vertical maps of \eqref{Diagram: Kseq} (with $R$, $F$, $E$ and $\Gamma$ replaced by $\ZZ$, $\QQ$, $\RR$ and $G$ respectively, and the obvious inclusion $\QQ \subseteq \RR$), and last equality is given by the result \cite[Prop. 3.1]{BB} of Bley and Burns (where the equivariant global epsilon constant $\mathcal{E}_{L/K}$ defined in \cite[p. 551]{BB} is equal to $\delta_G( \epsilon_{L/K})$). 

Turning to claim (ii), we write $\widehat{\delta}_{G_w, \ell}:=\partial^1_{\ZZ_\ell,\QQ_\ell,G_w}\circ ({\rm Nrd}_{\QQ_\ell[G_w]})^{-1}$ (as in \eqref{frak c local delta}). 

At the outset we note that, under the stated hypotheses, the decomposition subgroup $G_w$ in $G$ of any place $w$ of $L$ lying above $v \in \mathcal{W}_{L/K}$ is abelian, 
and so there is no irreducible symplectic character in $G_w$. Therefore, by Lemma~\ref{Prop: Nrd bijective cond}(ii)(b), one has $\zeta(\mathbb{Q}[G_w])^{\times+} = \zeta(\mathbb{Q}[G_w])^{\times}$.
	
We also note that, for all $v\in S_f(K)$ and the decomposition subgroup $G_w$ in $G$ of any place of $L$ above $v$,
\begin{equation}\label{image + implies}
x \in \zeta(\mathbb{Q}[G_w])^{\times+} \quad \implies \quad  \mathrm{\tilde{i}}_{G_{w}}^{G} (x) \in \zeta(\mathbb{Q}[G])^{\times+}.
\end{equation}
%
%To see this, we recall the explicit description of $\zeta(\mathbb{Q}[\Gamma])^{\times+}$ from Lemma~\ref{Prop: Nrd bijective cond}(iii) and \eqref{description of image nrd RR} (with $\Gamma=G$ and $G_w$), then the result follows from that, for each $\chi\in \widehat{G}$,
%the individual coefficient (by the definition of the induction map in \eqref{Eq: center ind def}) $\mathrm{\tilde{i}}_{G_w}^{G} (x)_{\chi}= \prod_{\varphi \in \widehat{G}_w} x_{\varphi}^{<\mathrm{res}^{G}_{G_w} \chi , \varphi >_{G_w}}$ is strictly positive (since $x \in \zeta(\mathbb{Q}[G_w])^{\times+}$ implies that $x_\varphi>0$ for all $\varphi\in\widehat{G}_w$). 
	
To prove claim (ii), we first claim that,  under the stated hypothesis, one has
\[\delta_{G}(\prod_{v \in \mathcal{W}_{L/K}} \tilde{{\rm i}}^{G}_{G_w} (y_{L_w/K_v}^{(2)})) = \mathfrak{c}_{L/K}. \]
	%
%where $\delta_{G}$ is the extended boundary homomorphism from Lemma~\ref{ext bound hom} and 
Clearly $\tilde{{\rm i}}^{G}_{G_w} (y_{L_w/K_v}^{(2)})$ belongs to $\zeta(\QQ[G])^{\times}$ (from \eqref{unram maximal order} and \eqref{Eq: center ind def}). To prove the claimed equality, it is enough for us to show that, under the natural projection map $j_{\ell, *}: K_0(\ZZ[G],\QQ[G]) \to K_0(\ZZ_\ell[G],\QQ_\ell[G])$ (given in \eqref{K_0 proj K_0_p}), the image in $K_0(\ZZ_\ell[G],\QQ_\ell[G])$ of the elements $\delta_{G}(\prod_{v \in \mathcal{W}_{L/K}} \tilde{{\rm i}}^{G}_{G_w} (y_{L_w/K_v}^{(2)}))$ and $\mathfrak{c}_{L/K}$   coincide. It follows that, for a fixed prime $\ell$,
\begin{align*}
&\, j_{\ell, *}\left(\delta_{G}\bigl(\prod_{v \in \mathcal{W}_{L/K}} \tilde{{\rm i}}^{G}_{G_w} (y_{L_w/K_v}^{(2)})\bigr) - \mathfrak{c}_{L/K} \right)\\
%=&\, \widehat{\delta}_{G_w, \ell} \bigl(  \prod_{v \in \mathcal{W}_{L/K}} \tilde{{\rm i}}^{G}_{G_w}( y_{L_w/K_v}^{(2)}) \bigr) - \sum_{v|\ell}{\rm i}^{G,*}_{G_w, \QQ_\ell} \bigl( \widehat{\delta}_{G_w, \ell}( y_{L_w/K_v}^{(2)}) \bigr)\\
=&\,  \sum_{v \in \mathcal{W}_{L/K}} {\rm i}^{G, *}_{G_w, \QQ_\ell} \big(\widehat{\delta}_{G_w, \ell} ( y_{L_w/K_v}^{(2)})   \big) 
- \sum_{v|\ell}{\rm i}^{G,*}_{G_w, \QQ_\ell} \bigl( \widehat{\delta}_{G_w, \ell}( y_{L_w/K_v}^{(2)}) \bigr)\\
%\delta_{G, \ell} \circ \tilde{{\rm i}}^{G}_{G_w} \left( \prod_{v \in \mathcal{W}_{L/K}} (y_{L_w/K_v}^{(2)}) \times  \prod_{v|\ell } (y_{L_w/K_v}^{(2)}) \right)\\
=&\, \sum_{v \in \mathcal{W}_{L/K}, v\nmid \ell} {\rm i}^{G, *}_{G_w, \QQ_\ell} \big(\widehat{\delta}_{G_w, \ell} ( y_{L_w/K_v}^{(2)})\big),
%\delta_{G, \ell} \circ \tilde{{\rm i}}^{G}_{G_w}\bigl(\prod_{v \in \mathcal{W}_{L/K}, v\nmid \ell} y_{L_w/K_v}^{(2)} \bigr).
\end{align*}
where the first equality follows from Lemma~\ref{ext bound hom}(iii), the commutative diagram \eqref{Diagram: K_0 ind commute} and Remark~\ref{Rem: idelic c l-comp}(ii). 
%noting here we don't need the fact that $j_{p, *}^c \circ {\rm i}^*_{w} ={\rm i}^*_{w,p} \circ j_{p, *}^c$ (directly from definitions \eqref{Eq: K_0 ind def} and \eqref{jlc def}) since it is given by the inclusion $\zeta(\QQ[G_w]) \subseteq \zeta(\QQ_\ell[G_w])$
	
It is therefore enough to show that, for each $v\in \mathcal{W}_{L/K}$, if $v$ does not divides $\ell$, the element $\widehat{\delta}_{G_w, \ell} ( y_{L_w/K_v}^{(2)})$ vanishes in $K_0(\ZZ_{\ell}[G_w], \QQ_\ell[G_w])$. To do this, we write $p$ for the residue characteristic of $v$ (and so $p\neq \ell$). Now, under the stated hypothesis, $G_w$ is abelian and the inertia subgroup $I_w:= G_{w, 0}$ of $G_w$ is a $p$-group, and so (since $\ZZ_{\ell}[I_w]$ is the maximal $\ZZ_{\ell}$-order in $\QQ_\ell[I_w]$) the proof of Lemma~\ref{Prop: local c tame vanishes} implies that 
\[ y_{L_w/K_v}^{(2)}=(1-e_{I_w})+\sigma_we_{I_w} \in {\rm Nrd}_{\QQ_\ell[G_w]}(\ZZ_{\ell}[G_w]^{\times}),\]
where where $\sigma_w$ is a lift to $G_w$ of the Frobenius element in $G_w/I_w$. Then, by the exactness of \eqref{Diagram: Kseq}, the element $\widehat{\delta}_{G_w, \ell} ( y_{L_w/K_v}^{(2)})$ vanishes in $K_0(\ZZ_{\ell}[G_w], \QQ_\ell[G_w])$, as required.
	
It follows that
\begin{align*}
\partial_{\ZZ,\QQ,G}^0 ( \mathfrak{c}_{L/K}) 
=&\, \partial_{\ZZ,\QQ,G}^0 \left(\delta_{G}(\prod_{v \in \mathcal{W}_{L/K}} \tilde{{\rm i}}^{G}_{G_w} (y_{L_w/K_v}^{(2)})) \right)\\
=&\, \partial_{\ZZ,\QQ,G}^0 \left(  \partial^{1}_{\ZZ,\QQ,G} (({\rm Nrd}_{\QQ[G]})^{-1}  (\prod_{v \in \mathcal{W}_{L/K}} \tilde{{\rm i}}^{G}_{G_w}(y_{L_w/K_v}^{(2)}))) \right)\\
=&\, 0.
\end{align*}	
Here, the second equality is given by \eqref{image + implies} and Lemma~\ref{ext bound hom}(ii), and the third one follows from the fact that the composite map $\partial_{\ZZ,\QQ,G}^0 \circ \partial_{\ZZ,\QQ,G}^1$ is zero. This completes the proof of our claim.	
\end{proof}

Now we are ready to prove Theorem~\ref{Lem: local ab Erez}. 

Under the stated hypothesis, claim (i) follows directly from Theorem~\ref{last intro thm}(i), Proposition~\ref{frak a global independence}(iv) and Lemma~\ref{Prop: local ab frak c proj}. Then, claim (ii) can be deduced immediately from claim (i) and the result \cite[Cor. 2]{BC} of Bley and Cobbe recalled below.
\begin{prop}[{Bley and Cobbe}]\label{Prop: chin conj BC}
	Let $L/K$ be a weakly ramified finite Galois extension of number fields that satisfies conditions (ii) - (vi) in Hypothesis \ref{big hyp}. Then, $\Omega(L/K, 2) = W_{L/K}$.
\end{prop}

\end{document}